\documentclass[11pt, a4paper]{article}

\usepackage[whole]{bxcjkjatype}

\usepackage{bm, amsmath, amsthm, amssymb, accents, comment}
\usepackage{ascmac}
\RequirePackage{amsthm,amsmath,amsfonts,amssymb}
\RequirePackage[numbers]{natbib}
\RequirePackage{graphicx}\usepackage{tikz}
\usepackage{pgfplots}
\usepackage{subcaption}
\pgfplotsset{
	compat=newest, 
	cycle list name=exotic }

\newtheorem{theorem}{Theorem}[section]
\newtheorem{proposition}{Proposition}[section]
\newtheorem{lemma}{Lemma}[section]
\newtheorem{corollary}{Corollary}[section]
\newcommand{\relmiddle}[1]{\mathrel{}\middle#1\mathrel{}}

\DeclareMathOperator*{\argmax}{arg\,max}

\usepackage{typearea}
\typearea{12}

\begin{document}
	
	\title{Adapting to general quadratic loss via singular value shrinkage}	
	
	\author{Takeru Matsuda\thanks{Department of Mathematical Informatics, The University of Tokyo \& Statistical Mathematics Unit, RIKEN Center for Brain Science, e-mail: \texttt{matsuda@mist.i.u-tokyo.ac.jp}}}
	
	\date{}

\maketitle

\begin{abstract}
The Gaussian sequence model is a canonical model in nonparametric estimation.
In this study, we introduce a multivariate version of the Gaussian sequence model and investigate adaptive estimation over the multivariate Sobolev ellipsoids, where adaptation is not only to unknown smoothness but also to arbitrary quadratic loss. 
First, we derive an oracle inequality for the singular value shrinkage estimator by Efron and Morris, which is a matrix generalization of the James--Stein estimator. 	Next, we develop an asymptotically minimax estimator on the multivariate Sobolev ellipsoid for each quadratic loss, which can be viewed as a generalization of Pinsker's theorem.
Then, we show that the blockwise Efron--Morris estimator is exactly adaptive minimax over the multivariate Sobolev ellipsoids under the corresponding quadratic loss. 
It attains sharp adaptive estimation of any linear combination of the mean sequences simultaneously.
\end{abstract}
	
	\section{Introduction}
Suppose that we observe
\begin{align*}
	y_i = \theta_i + \varepsilon \xi_i, \quad i=1,2,\cdots,
\end{align*}
where $\varepsilon>0$ and $\xi_{i} \sim {\rm N}_p (0, I_p)$ are independent $p$-dimensional Gaussian random vectors.
This is a multivariate version of the Gaussian sequence model, which is a canonical model in nonparametric estimation \citep{Johnstone,Tsybakov,Wasserman}.
We consider estimation of $\theta=(\theta_{i})$ under the quadratic loss specified by a $p \times p$ positive definite symmetric matrix $Q$:
\begin{align*}
	L_Q(\theta,\hat{\theta}) = \sum_{i=1}^{\infty} (\hat{\theta}_i-\theta_i)^{\top} Q (\hat{\theta}_i - \theta_i).
\end{align*}
Let $R_Q(\theta,\hat{\theta}) = {\rm E}_{\theta} [L_Q(\theta,\hat{\theta})]$ be the risk function.
For a parameter space $\Theta$, an estimator $\hat{\theta}_*$ is said to be asymptotically minimax on $\Theta$ if
\begin{align*}
	\sup_{\theta \in \Theta} R_Q(\theta,\hat{\theta}_*) \sim \inf_{\hat{\theta}} \sup_{\theta \in \Theta} R_Q(\theta,\hat{\theta})
\end{align*}
as $\varepsilon \to 0$, where $a \sim b$ denotes $a/b \to 1$ and $\inf$ is taken over all estimators.
For a class of parameter spaces $\mathcal{C}= \{ \Theta \}$, an estimator $\hat{\theta}_*$ is said to be adaptive minimax over $\mathcal{C}$ if it is asymptotically minimax on every $\Theta$ in $\mathcal{C}$.
Note that we focus on exact minimaxity in this study, which is stronger than minimaxity in terms of the rate of convergence.

The above problem has been well studied in the case of $p=1$ \cite{Johnstone,Cai12}, where we can fix $Q=1$ without loss of generality.
It is asymptotically equivalent to various problems such as nonparametric regression \cite{Brown} and density estimation \cite{Nussbaum}.
In these settings, each $\theta_i$ corresponds to the coefficient in basis expansion of an unknown function $f$, and smoothness constraints on $f$ are translated to domain constraints on $\theta=(\theta_i)$.
Among them, the Sobolev class is represented by infinite-dimensional ellipsoids under the Fourier expansion.
For such ellipsoids, Pinsker \cite{Pinsker} showed that a linear estimator attains asymptotic minimaxity and derived its concrete form.
Pinsker's estimator for the Sobolev ellipsoid depends on its smoothness and scale parameters, which are generally unknown. 
Thus, estimators that are adaptive minimax over the Sobolev ellipsoids have been developed.
Among them, the blockwise Jame--Stein estimator \cite{Efromovich} partitions the observations into blocks and apply the James--Stein shrinkage to each block.
Adaptive minimaxity of this estimator is proved by using an oracle inequality for the James--Stein estimator, which indicates that the James--Stein estimator attains almost the same risk with the best linear estimator.
More details will be given in Section~\ref{sec_pre}.

In this study, we generalize the above results to $p \geq 2$, which corresponds to multivariate versions of nonparametic problems such as nonparametric regression with vector response.
Unlike the case of $p=1$, there is an arbitrariness in the choice of the positive definite matrix $Q$ in the quadratic loss. 
We introduce a multivariate version of the Sobolev ellipsoid, and develop an estimator that attains adaptive minimaxity over the multivariate Sobolev ellipsoids under any choice of $Q$ simultaneously.
In this sense, our estimator is adaptive not only to smoothness but also to arbitrary quadratic loss.
Note that, while minimax estimation under arbitrary quadratic loss has been studied in the finite-dimensional (parametric) setting as well \cite{Berger,Strawderman}, these previous studies proposed minimax estimators that depend on $Q$. 
In contrast, we develop an estimator that attains adaptive minimaxity under any quadratic loss simultaneously.
Note that applying the blockwise James--Stein estimator for each of the $p$ components only attains adaptive minimaxity for diagonal $Q$.

Our estimator is based on the Efron--Morris estimator \cite{Efron72}, which was originally developed as an empirical Bayes estimator of a normal mean matrix.
Efron and Morris \cite{Efron72} showed that this estimator is minimax and dominates the maximum likelihood estimator under the Frobenius loss.
The Efron--Morris estimator shrinks the singular values towards zero and can be viewed as a matrix generalization of the James--Stein estimator.
It attains large risk reduction when the truth is close to low rank.
Recently, the Efron--Morris estimator has been shown to dominate the maximum likelihood estimator even under the matrix quadratic loss, which is a matrix-valued loss function suitable for matrix estimation \cite{Matsuda22}.
This result implies that the Efron--Morris estimator attains improved estimation of any linear combination of the column vectors simultaneously.
In other words, it is adaptive to arbitrary quadratic loss.
More details will be given in Section~\ref{sec_em}.

This paper is organized as follows.
In Section~\ref{sec_pre}, we briefly review existing results on adaptive estimation in the Gaussian sequence model.
In Section~\ref{sec_em}, we introduce the Efron--Morris estimator and derive its oracle inequality.
In Section~\ref{sec_seq}, we define the multivariate Gaussian sequence model and multivariate Sobolev ellipsoid.
In Section~\ref{sec_pinsker}, we derive an asymptotically minimax estimator on the multivariate Sobolev ellipsoid.
In Section~\ref{sec_bem}, we prove the adaptive minimaxity of the blockwise Efron--Morris estimator over the multivariate Sobolev ellipsoids.
In Section~\ref{sec_simulation}, we present simulation results.
In Section~\ref{sec_concl}, we give concluding remarks.

Our primary contribution is threefold: (1) We introduce a multivariate version of the Gaussian sequence model and Soboloev ellipsoid, (2) We derive an oracle inequality for the Efron--Morris estimator, (3) We show that the blockwise Efron--Morris estimator attains adaptive minimaxity over the multivariate Sobolev ellipsoids.
To the best of our knowledge, adaptive estimation in the multivariate Gaussian sequence model has not been investigated so far.

Throughout the paper, we write $a \sim b$ when $a/b \to 1$. 
For a real number $a$, we define $(a)_+=\max(0,a)$.
For a symmetric matrix $S \in \mathbb{R}^{p \times p}$ that has a spectral decomposition $S=U^{\top} \Lambda U$ with $U \in O(p)$ and $\Lambda={\rm diag} (\lambda_1,\dots,\lambda_p)$, we write its projection onto the positive semidefinite cone by $S_+=U^{\top} \Lambda_+ U \succeq O$, where $\Lambda_+={\rm diag} ((\lambda_1)_+,\dots,(\lambda_p)_+)$.
For two symmetric matrices $A \in \mathbb{R}^{p \times p}$ and $B \in \mathbb{R}^{p \times p}$, we write $A \preceq B$ when $B-A$ is positive semidefinite (L\"{o}wner order).
For a vector-valued function $f: [0,1] \to \mathbb{R}^p$, we write each component of $f$ as $f_j: [0,1] \to \mathbb{R}$ for $j=1,\dots,p$.
For $\beta \in \{ 1,2,\cdots \}$, we define $f^{(\beta)}: [0,1] \to \mathbb{R}^p$ by $f^{(\beta)} (x)=(f_1^{(\beta)} (x),\dots,f_p^{(\beta)} (x))^{\top}$ where $f_j^{(\beta)}$ is the $\beta$-th derivative of $f_j$.

\section{Preliminaries}\label{sec_pre}	
\subsection{Gaussian sequence model}
The Gaussian sequence model \citep{Johnstone,Tsybakov} is defined as
\begin{align}
	y_i = \theta_i + \varepsilon \xi_i, \quad i=1,2,\cdots, \label{sequence}
\end{align}
where $\varepsilon>0$ and $\xi_i \sim {\rm N} (0,1)$ are independent standard Gaussian random variables.
This model is equivalent to the Gaussian white noise model, which is defined by the stochastic differential equation
\begin{align*}
	{\rm d} Y(t) = f(t) {\rm d} t + \varepsilon {\rm d} W(t), \quad t \in [0,1],
\end{align*}
where $f(t)$ is a drift function to be estimated and $W(t)$ is the standard Wiener process.
By using an orthonormal basis $\phi=(\phi_i)$ of $L_2 [0,1]$ such as the Fourier basis, let
\begin{align*}
	y_i = \int_0^1 \phi_i(t) {\rm d} Y(t), \quad \theta_i = \int_0^1 \phi_i(t) f (t) {\rm d} t, \quad \xi_i = \int_0^1 \phi_i(t) {\rm d} W (t).
\end{align*}
Then, $y=(y_i)$ follows the Gaussian sequence model \eqref{sequence}.
In addition, the Gaussian nonparametric regression model
\begin{align*}
	y_i = f(t_i) + \xi_i, \quad i=1,\dots,n,
\end{align*}
where $f:[0,1] \to \mathbb{R}$, $t_i=i/n$ and $\xi_i \sim {\rm N} (0,1)$ are independent, is asymptotically equivalent to the Gaussian sequence model \eqref{sequence} with $\varepsilon = n^{-1/2}$ as $n \to \infty$ \citep{Brown} in the sense of Le Cam's limits of experiments \cite{Gine}.
Asymptotic equivalence to other problems such as density estimation  and Poisson processes has been also established \cite{Nussbaum,Brown04}.
Thus, the Gaussian sequence model \eqref{sequence} is a canonical model in nonparametric statistics and many studies have investigated the problem of estimating $\theta=(\theta_i)$ from $y=(y_i)$ \citep{Cai12,Johnstone,Tsybakov}.

\subsection{Sobolev ellipsoid and Pinsker's theorem}
In nonparametric statistics, unknown functions are often assumed to belong to some smoothness class.
Among them, the periodic Sobolev class on $[0,1]$ is defined as
\begin{align*}
	W (\beta,L) :=  \left\{ f \in L_2[0,1] \relmiddle| \right. & f^{(k)} (0) = f^{(k)} (1), \ k=0,\cdots,\beta-1, \\
	& \left. f^{(\beta - 1)} : {\rm absolutely\ continuous}, \ \int_0^1 (f^{(\beta)} (x))^2 {\rm d} x \leq L^2 \right\},
\end{align*}
where $\beta \in \{ 1,2,\cdots \}$, $L>0$, $L_2[0,1]$ is the space of real-valued square-integrable functions on $[0,1]$ and $f^{(k)}$ denotes the $k$-th order derivative of $f$.
It has a useful representation using the Fourier coefficients as follows \citep{Johnstone,Tsybakov}.

\begin{lemma}\label{lem_sobolev} 
	For $\theta=(\theta_i)$, let
	\begin{align*}
		f (x) = \sum_{i=1}^{\infty} \theta_{i} \phi_i(x),
	\end{align*}
	where $\phi=(\phi_i)$ is the Fourier basis of $L_2 [0,1]$ given by
	\begin{align}
		\phi_1(x) \equiv 1, \ \phi_{2k}(x) = \sqrt{2} \cos (2 \pi k x), \ \phi_{2k+1}(x) = \sqrt{2} \sin (2 \pi k x), \ k=1,2,\dots. \label{fourier}
	\end{align}
	Then, $f \in W (\beta,L)$ if and only if 
	\begin{align}
		\theta \in \Theta (\beta,R) := \left\{ \theta = (\theta_1,\theta_2,\dots) \relmiddle| \sum_{i=1}^{\infty} a_{\beta,i}^2 \theta_i^2 \leq R \right\}, \label{ellipsoid}
	\end{align}
	where $R:=L^2/\pi^{2 \beta}$ and
	\begin{align}\label{sobolev_coeff}
		a_{\beta,i} := \begin{cases} i^{\beta}, & i : {\rm even}, \\ (i-1)^{\beta}, & i : {\rm odd}. \end{cases}
	\end{align}
\end{lemma}

From Lemma~\ref{lem_sobolev}, the infinite-dimensional ellipsoid $\Theta (\beta,R)$ in \eqref{ellipsoid} is called the Sobolev ellipsoid for $\beta>0$ and $R>0$.

Now, consider the problem of estimating $\theta=(\theta_i)$ from $y=(y_i)$ in the Gaussian sequence model \eqref{sequence} under the quadratic loss
\begin{align}
	L(\theta,\hat{\theta}) = \| \hat{\theta}-\theta \|^2 = \sum_{i=1}^{\infty} (\hat{\theta}_i-\theta_i)^2, \label{qloss}
\end{align}
and assume that $\theta$ is restricted to the Sobolev ellipsoid $\Theta(\beta,R)$ in \eqref{ellipsoid}.
Pinsker \citep{Pinsker} studied this problem under the asymptotics $\varepsilon \to 0$ and derived an asymptotically minimax estimator as follows.

\begin{proposition}(Pinsker's theorem for Sobolev ellipsoid)\label{prop_pinsker}
	Let $C_{\mathrm{P},i} = (1-a_{\beta,i} \kappa R)_+$ for $i=1,2,\dots$, where $(a)_+=\max(0,a)$ and $\kappa>0$ is the unique solution of
	\begin{align*}
		\varepsilon^2 \kappa^{-1} \sum_{i=1}^{\infty} a_{\beta,i} (1-a_{\beta,i} \kappa R^{-1})_+ = 1.
	\end{align*}
Then, the linear estimator $\hat{\theta}_{\rm P}=(\hat{\theta}_{\mathrm{P},i})$ with $\hat{\theta}_{\mathrm{P},i} = C_{\mathrm{P},i} y_i$ is asymptotically minimax on the Sobolev ellipsoid $\Theta(\beta,R)$:
	\begin{align*}
		\sup_{\theta \in \Theta(\beta,R)} {\rm E}_{\theta} [ \| \hat{\theta}_{\rm P}-\theta \|^2 ] \sim \inf_{\hat{\theta}} \sup_{\theta \in \Theta(\beta,R)} {\rm E}_{\theta} [ \| \hat{\theta}-\theta \|^2 ] \sim P(\beta,R) \varepsilon^{{4 \beta}/{(2 \beta+1)}}
	\end{align*}
	as $\varepsilon \to 0$, where $\inf$ is taken over all estimators and
	\begin{align*}
		P(\beta,R) := (R (2 \beta+1))^{1/(2 \beta+1)} \left( \frac{\beta}{\beta+1} \right)^{2 \beta/(2 \beta+1)}.
	\end{align*}
\end{proposition}

\subsection{Blockwise James--Stein estimator}\label{sec_bjs}
Pinsker's estimator in Proposition~\ref{prop_pinsker} depends on the smoothness parameter $\beta$ and scale parameter $R$ of the Sobolev ellipsoid $\Theta(\beta,R)$ in \eqref{ellipsoid}, which are difficult to specify in general.
Thus, adaptive minimax estimators over the Sobolev ellipsoids have been developed, which are asymptotically minimax on $\Theta(\beta,R)$ for every $\beta$ and $R$ simultaneously.
Among them, here we focus on the blockwise James--Stein estimator \cite{Efromovich,Tsybakov}.

In estimation of $\mu$ from $X \sim {\rm N}_n (\mu,I_n)$ under the quadratic loss, the maximum likelihood estimator $\hat{\mu}_{{\rm ML}}=x$ is minimax \cite{Lehmann}.
However, Stein \citep{Stein74} showed that it is inadmissible when $n \geq 3$ and James and Stein \citep{James} proved that the shrinkage estimator
\begin{align*}
	\hat{\mu}_{\rm JS} (x) = \left( 1-\frac{n-2}{\| x \|^2} \right) x
\end{align*}
is minimax and dominates the maximum likelihood estimator:
\begin{align*}
	{\rm E}_{\mu} [ \| \hat{\mu}_{{\rm JS}}-\mu \|^2 ] \leq {\rm E}_{\mu} [ \| \hat{\mu}_{{\rm ML}}-\mu \|^2 ] = n.
\end{align*}
The James--Stein estimator also satisfies the following oracle inequality \cite{Johnstone,Tsybakov}.

\begin{lemma}\label{lem_js}
	Let $\hat{\mu}_{C} = C x$ be the linear estimator of $\mu$ for $C \in \mathbb{R}$.
	Then, \begin{align}
		\min_{C} {\rm E}_{\mu} [ \| \hat{\mu}_{C}-\mu \|^2 ] \leq {\rm E}_{\mu} [ \| \hat{\mu}_{{\rm JS}}-\mu \|^2 ] \leq \min_{C} {\rm E}_{\mu} [ \| \hat{\mu}_{C}-\mu \|^2 ] + 2, \label{js_oracle}
	\end{align}
	for every $\mu$.
\end{lemma}

The oracle inequality \eqref{js_oracle} indicates that the James--Stein estimator attains almost the same risk with the best linear estimator.
Then, since Pinsker's estimator in Proposition~\ref{prop_pinsker} is a linear estimator, it is expected that an estimator that is asymptotically minimax on every Sobolev ellipsoid can be constructed by using the James--Stein shrinkage.
This idea leads to the blockwise James--Stein estimator for the Gaussian sequence model \eqref{sequence} as follows.

Consider a partition of $\{ 1,2,\dots, N \}$ with blocks
\begin{align*}
	B_j := \{ l_{j-1}+1, l_{j-1}+2,\dots, l_j \}, \quad j=1,\dots,J,
\end{align*}
where $0=l_0 < l_1 < \dots < l_J=N$.
The blockwise James--Stein estimator is defined as
\begin{align}
	(\hat{\theta}_{\mathrm{BJS}} (y))_i = \begin{cases} y_i, & i \in B_j, \ |B_j| < 3, \\ \left( 1-\frac{(|B_j|-2) \varepsilon^2}{\| y_{(j)}\|^2} \right) y_i, & i \in B_j, \ |B_j| \geq 3, \\ 0, & i > N, \end{cases} \label{BJS}
\end{align}
where $\| y_{(j)} \|^2 := \sum_{i \in B_j} y_i^2$.
This estimator applies the James--Stein shrinkage to each block $B_j$.
In this study, we focus on the weakly geometric blocks defined by $N := \lfloor \varepsilon^{-2} \rfloor$ and \begin{align*}
	l_j := \lfloor \rho_{\varepsilon}^{-1} (1+\rho_{\varepsilon})^{j} \rfloor, \quad j=1,\dots,J-1, \end{align*}
where $\rho_{\varepsilon}:=(\log (1/\varepsilon))^{-1}$ and $J:=\min \{ j \mid \rho_{\varepsilon}^{-1} (1+\rho_{\varepsilon})^{j} \geq \varepsilon^{-2} \}$.
The blockwise James--Stein estimator with the weakly geometric blocks attains adaptive minimaxity as follows \cite{Efromovich,Tsybakov}.

\begin{proposition}\label{prop_wgb}
	For every $\beta$ and $R$, the blockwise James--Stein estimator $\hat{\theta}_{\rm BJS}$ with the weakly geometric blocks is asymptotically minimax on the Sobolev ellipsoid $\Theta(\beta,R)$: 
	\begin{align*}
		\sup_{\theta \in \Theta(\beta,R)} {\rm E}_{\theta} [ \| \hat{\theta}_{\mathrm{BJS}}-\theta \|^2 ] \sim \inf_{\hat{\theta}} \sup_{\theta \in \Theta(\beta,R)} {\rm E}_{\theta} [ \| \hat{\theta}-\theta \|^2 ] \sim P(\beta,R) \varepsilon^{{4 \beta}/{(2 \beta+1)}}
	\end{align*}
	as $\varepsilon \to 0$, where $\inf$ is taken over all estimators.
\end{proposition}

\section{Efron--Morris estimator and its oracle inequality}\label{sec_em}	
\subsection{Efron--Morris estimator}
Suppose that we have a matrix observation $X \in \mathbb{R}^{n \times p}$ whose entries are independent Gaussian random variables $X_{ij} \sim {\rm N} (M_{ij},1)$, where $n-p-1>0$ and $M \in \mathbb{R}^{n \times p}$ is an unknown mean matrix. 	
In this setting, Efron and Morris \cite{Efron72} considered estimation of $M$ under the Frobenius loss
\begin{align*}
	l(M,\hat{M}) = \| \hat{M} - M \|_{{\rm F}}^2 = \sum_{i=1}^n \sum_{j=1}^p (\hat{M}_{ij} - M_{ij})^2,
\end{align*}
and proposed an estimator
\begin{align}
	\hat{M}_{{\rm EM}} = X ( I_p-(n-p-1) (X^{\top} X)^{-1} ). \label{EM_estimator}
\end{align}
It coincides with the James--Stein estimator when $p=1$.
They showed that $\hat{M}_{{\rm EM}}$ is minimax and dominates the maximum likelihood estimator $\hat{M}_{\rm ML}=X$ under the Frobenius loss:
\begin{align*}
	{\rm E}_M [ \| \hat{M}_{\rm EM} - M \|_{{\rm F}}^2 ] \leq {\rm E}_M [ \| \hat{M}_{\rm ML} - M \|_{{\rm F}}^2 ] = np.
\end{align*}
The Efron--Morris estimator $\hat{M}_{{\rm EM}}$ can be interpreted as an empirical Bayes estimator, in which an independent prior ${\rm N}_p(0,\Sigma)$ is put on each row of $M$ and the hyperparameter $\Sigma$ is estimated from $X$.

Stein \cite{Stein74} pointed out that $\hat{M}_{{\rm EM}}$ does not change the singular vectors but shrinks the singular values of $X$ towards zero.
Namely, let $X=U \Sigma V^{\top}$ be the singular value decomposition of $X$, where $U^{\top} U = V^{\top} V = I_p$ and $\Sigma = {\rm diag} (\sigma_1(X),\dots,\sigma_p(X))$.
Then, the singular value decomposition of $\hat{M}_{{\rm EM}}$ is given by $\hat{M}_{{\rm EM}} = U \widetilde{\Sigma} V^{\top}$, where $\widetilde{\Sigma} = {\rm diag} (\sigma_1(\hat{M}_{{\rm EM}}),\dots,\sigma_p(\hat{M}_{{\rm EM}}))$ and
\begin{align*}
	\sigma_j(\hat{M}_{{\rm EM}}) = \left( 1 - \frac{n-p-1}{\sigma_j(X)^2} \right) \sigma_j(X), \quad j=1, \dots, p.
\end{align*}
From this property, $\hat{M}_{{\rm EM}}$ attains large risk reduction when $M$ has small singular values.
In particular, $\hat{M}_{{\rm EM}}$ works well when $M$ is close to low rank.

Figure~\ref{EMrisk} compares the Frobenius risk of $\hat{M}_{{\rm EM}}$ with that of $\hat{M}_{{\rm JS}}(X)=(1-(np-2)/\| X \|_{\mathrm{F}}^2)X$, which applies the James--Stein shrinkage to the vectorization of $X$.
It indicates that $\hat{M}_{{\rm EM}}$ attains constant risk reduction when some singular values of $M$ are small, regardless of the magnitude of the other singular values.
Thus, $\hat{M}_{{\rm EM}}$ works well for low rank matrices and this advantage is more pronounced in higher dimensions \cite{Matsuda22}.
On the other hand, $\hat{M}_{{\rm JS}}$ works well only when $\| M \|_{\mathrm{F}}^2=\sum_{j=1}^p \sigma_j(M)^2$ is small.

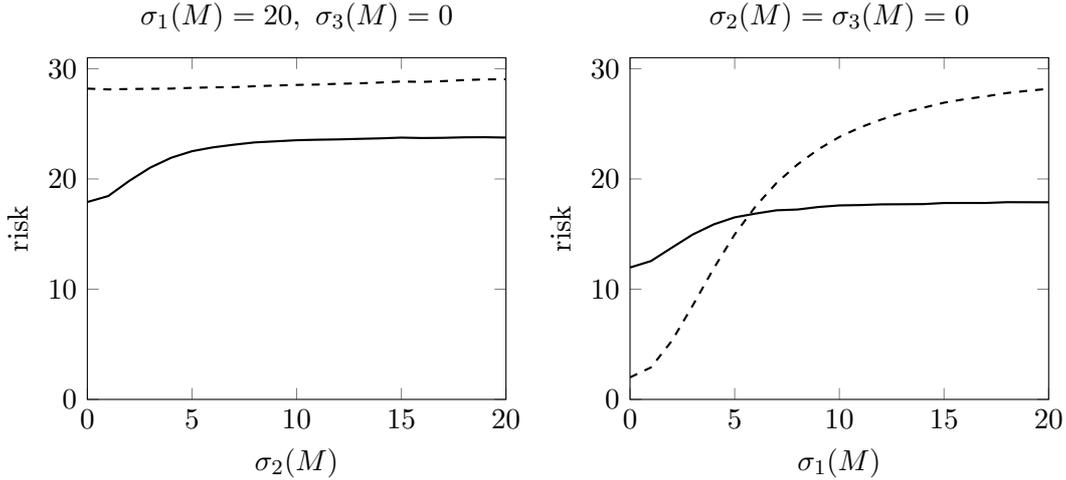
\begin{figure}[h]
	\centering
	\begin{tikzpicture}
		\begin{axis}[
			title={$\sigma_1(M)=20, \ \sigma_3(M)=0$},
			xlabel={$\sigma_2(M)$}, xmin=0, xmax=20,
			ylabel={risk}, ymin=0, ymax=31,
			width=0.45\linewidth
			]
			\addplot[thick, color=black,
			filter discard warning=false, unbounded coords=discard
			] table {
				0   17.9076
				1.0000   18.4515
				2.0000   19.8225
				3.0000   21.0238
				4.0000   21.9198
				5.0000   22.5189
				6.0000   22.8671
				7.0000   23.1092
				8.0000   23.3180
				9.0000   23.4152
				10.0000   23.5154
				11.0000   23.5604
				12.0000   23.5896
				13.0000   23.6394
				14.0000   23.6841
				15.0000   23.7518
				16.0000   23.7142
				17.0000   23.7320
				18.0000   23.7780
				19.0000   23.7865
				20.0000   23.7563
			};
			\addplot[dashed, thick, color=black,
			filter discard warning=false, unbounded coords=discard
			] table {
				0   28.1968
				1.0000   28.1222
				2.0000   28.1633
				3.0000   28.1788
				4.0000   28.1984
				5.0000   28.2664
				6.0000   28.3138
				7.0000   28.3293
				8.0000   28.4149
				9.0000   28.4765
				10.0000   28.5322
				11.0000   28.5645
				12.0000   28.6374
				13.0000   28.6762
				14.0000   28.7490
				15.0000   28.8404
				16.0000   28.8067
				17.0000   28.8744
				18.0000   28.9737
				19.0000   29.0319
				20.0000   29.0449
			};
		\end{axis}
	\end{tikzpicture} 
	\begin{tikzpicture}
		\begin{axis}[
			title={$\sigma_2(M)=\sigma_3(M)=0$},
			xlabel={$\sigma_1(M)$}, xmin=0, xmax=20,
			ylabel={risk}, ymin=0, ymax=31,
			width=0.45\linewidth
			]
			\addplot[thick, color=black,
			filter discard warning=false, unbounded coords=discard
			] table {
				0   11.9638
				1.0000   12.5520
				2.0000   13.7688
				3.0000   14.9642
				4.0000   15.8799
				5.0000   16.5171
				6.0000   16.8603
				7.0000   17.1577
				8.0000   17.2277
				9.0000   17.4562
				10.0000   17.5997
				11.0000   17.6307
				12.0000   17.6948
				13.0000   17.7059
				14.0000   17.7182
				15.0000   17.8189
				16.0000   17.8218
				17.0000   17.8212
				18.0000   17.8935
				19.0000   17.8879
				20.0000   17.8848
			};
			\addplot[dashed, thick, color=black,
			filter discard warning=false, unbounded coords=discard
			] table {
				0    1.9988
				1.0000    2.9021
				2.0000    5.3163
				3.0000    8.5374
				4.0000   11.8932
				5.0000   14.9668
				6.0000   17.5222
				7.0000   19.6378
				8.0000   21.3152
				9.0000   22.6927
				10.0000   23.8099
				11.0000   24.6813
				12.0000   25.3877
				13.0000   25.9878
				14.0000   26.4560
				15.0000   26.9210
				16.0000   27.2408
				17.0000   27.4971
				18.0000   27.7964
				19.0000   28.0009
				20.0000   28.1781
			};
		\end{axis}
	\end{tikzpicture} 
	\caption{Frobenius risk of the Efron--Morris estimator (solid) and the James--Stein estimator (dashed) for $n=10$ and $p=3$. Left: $\sigma_1(M)=20$, $\sigma_3(M)=0$. Right: $\sigma_2(M)=\sigma_3(M)=0$}
	\label{EMrisk}
\end{figure}

\subsection{Matrix quadratic loss}
Instead of the Frobenius loss, it is insightful to investigate estimation of $M$ under the {matrix quadratic loss} 
\begin{align}
	L(M,\hat{M}) = (\hat{M} - M)^{\top} (\hat{M} - M), \label{mat_loss}
\end{align}
which takes a value in the set of $p \times p$ positive semidefinite matrices \cite{Matsuda22}.
Under this loss, an estimator $\hat{M}_1$ is said to dominate another estimator $\hat{M}_2$ if 
\begin{align*}
	{\rm E}_M [ (\hat{M}_1 - M)^{\top} (\hat{M}_1 - M) ] \preceq {\rm E}_M [ (\hat{M}_2 - M)^{\top} (\hat{M}_2 - M) ]
\end{align*}
for every $M$, where $\preceq$ is the L\"{o}wner order: $A \preceq B$ means that $B-A$ is positive semidefinite.
Thus, if $\hat{M}_1$ dominates $\hat{M}_2$ under the matrix quadratic loss, then
\begin{align*}
	{\rm E}_M [ \| (\hat{M}_1 - M) u \|^2 ] \leq {\rm E}_M [ \| (\hat{M}_2 - M) u \|^2 ]
\end{align*}
for every $M$ and $u \in \mathbb{R}^p$.
In other words, $\hat{M}_1$ dominates $\hat{M}_2$ in estimating any linear combination of the columns of $M$ under the quadratic loss.
In particular, each column of $\hat{M}_1$ dominates that of $\hat{M}_2$ as an estimator of the corresponding column of $M$ under the quadratic loss.
Domination under the matrix quadratic loss is stronger than domination under the Frobenius loss, because the Frobenius loss is equal to the trace of the matrix quadratic loss: $l(M,\hat{M}) = {\rm tr} (L(M,\hat{M}))$.

The Efron--Morris estimator dominates the maximum likelihood estimator even under the {matrix quadratic loss} as follows \cite{Matsuda22}.

\begin{lemma}When $n-p-1>0$, the matrix quadratic risk of the Efron--Morris estimator $\hat{M}_{{\rm EM}}$ in \eqref{EM_estimator} is 
	\begin{align}\label{EM_risk}
		{\rm E}_M [(\hat{M}_{{\rm EM}} - M)^{\top} (\hat{M}_{{\rm EM}} - M)] = n I_p - (n-p-1)^2 {\rm E}_M [(X^{\top} X)^{-1}] \preceq n I_p.
	\end{align}
	Thus, the Efron--Morris estimator dominates the maximum likelihood estimator under the matrix quadratic loss.
\end{lemma}

The matrix quadratic risk \eqref{EM_risk} of $\hat{M}_{{\rm EM}}$ at $M=O$ is $(p+1) I_p$ \cite{Matsuda22}.
Here, we extend this result by using the recent result by \cite{Hillier} on the expectation of the inverse of a noncentral Wishart matrix.

\begin{proposition}\label{prop_em}
	The matrix quadratic risk of the Efron--Morris estimator $\hat{M}_{{\rm EM}}$ in \eqref{EM_estimator} has the same eigenvectors with $M^{\top} M$.
	If $M^{\top} M$ has a zero eigenvalue, then the corresponding eigenvalue of the matrix quadratic risk is $p+1$.
\end{proposition}
\begin{proof}
	The matrix $X^{\top} X$ follows the noncentral Wishart distribution $W_p (n, I_p, M^{\top}M)$ \cite{Gupta}.
	Let $M^{\top} M = U^{\top} \Lambda U$ be a spectral decomposition of $M^{\top} M$, where $U \in O(p)$ and $\Lambda={\rm diag}(\sigma_1(M)^2,\dots,\sigma_p(M)^2)$.
	Then, from Theorem 1 of \cite{Hillier}, ${\rm E}_M [(X^{\top} X)^{-1}] = U^{\top} \Psi U$,where $\Psi={\rm diag} (\psi_1(M),\dots,\psi_p(M))$ and each $\psi_j(M)$ is given by an infinite sum involving matrix-variate hypergeometric functions.
	Therefore, the matrix quadratic risk \eqref{EM_risk} of $\hat{M}_{{\rm EM}}$ is given by
	\begin{align}
		{\rm E}_M [(\hat{M}_{{\rm EM}} - M)^{\top} (\hat{M}_{{\rm EM}} - M)] = U^{\top} (n I_p - (n-p-1)^2 \Psi) U, \label{EM_risk2}
	\end{align}
	which has the same eigenvectors with $M^{\top} M= U^{\top} \Lambda U$.
	If the $j$-th eigenvalue of $M^{\top} M$ is zero, then $\psi_j(M) = (n-p-1)^{-1}$ from Corollary 3 of \cite{Hillier}, which means that the corresponding eigenvalue of \eqref{EM_risk2} is $n-(n-p-1)=p+1$.
\end{proof}

Therefore, the Efron--Morris estimator works well when $M$ is close to low-rank.
Specifically, if $M^{\top} M$ has a zero eigenvalue, then the ratio of the corresponding eigenvalues of the matrix quadratic risks of the Efron--Morris estimator and the maximum likelihood estimator is $(p+1)/n$, which goes to zero as $n$ increases.
Thus, the advantage of the Efron--Morris estimator over the maximum likelihood estimator for low-rank matrices is pronounced when the number of row is much larger than the number of columns.

\subsection{Oracle inequality}
Now, we derive an oracle inequality for the Efron--Morris estimator, which will be used to show the adaptive minimaxity of the blockwise Efron--Morris estimator.

The matrix quadratic risk of the linear estimator $\hat{M}_C = XC$ is
\begin{align*}
	{\rm E}_M [(\hat{M}_C - M)^{\top} (\hat{M}_C - M)] = C^{\top} (M^{\top} M + n I_p) C - 2 M^{\top} M C + M^{\top} M.
\end{align*}
For a fixed $M$, it is uniquely minimized by $C_*=C_*(M)=(M^{\top} M + n I_p)^{-1} M^{\top} M$:
\begin{align*}
	{\rm E}_M [(\hat{M}_{C_*} - M)^{\top} (\hat{M}_{C_*} - M)] &= n I_p - n^2 (M^{\top} M + n I_p)^{-1} \\
	&\preceq {\rm E}_M [(\hat{M}_C - M)^{\top} (\hat{M}_C - M)]
\end{align*}
for every $C$.
Thus, $\hat{M}_{C_*(M)}$ can be viewed as the linear oracle estimator.
The Efron--Morris estimator $\hat{M}_{{\rm EM}}$ attains almost the same matrix quadratic risk with this oracle as follows.

\begin{theorem}\label{th_oracle}
	The matrix quadratic risk of the Efron--Morris estimator $\hat{M}_{{\rm EM}}$ satisfies
	\begin{align}
		{\rm E}_M [(\hat{M}_{{\rm EM}} - M)^{\top} (\hat{M}_{{\rm EM}} - M)] \preceq {\rm E}_M [(\hat{M}_{C_*} - M)^{\top} (\hat{M}_{C_*} - M)] + 2 (p+1) I_p \label{em_oracle}
	\end{align}
	for every $M$, where $C_*=C_*(M)=(M^{\top} M + n I_p)^{-1} M^{\top} M$.
\end{theorem}
\begin{proof}
	From \eqref{EM_risk},
	\begin{align*}
		{\rm E}_M [(\hat{M}_{{\rm EM}} - M)^{\top} (\hat{M}_{{\rm EM}} - M)] &= n I_p - (n-p-1)^2 {\rm E}_M [(X^{\top} X)^{-1}] \\
		& \preceq n I_p - (n-p-1)^2 {\rm E}_M [X^{\top}X]^{-1} \\
		& = n I_p - (n-p-1)^2 (M^{\top} M + n I_p)^{-1},
	\end{align*}
	where we used the operator convexity of the inverse function \citep{Groves,Hiai}.
	Then, from $(n-p-1)^2 \geq n^2 -2n(p+1)$ and $(M^{\top} M + n I_p)^{-1} \preceq n^{-1} I_p$, we obtain \eqref{em_oracle}. 
\end{proof}

Consider the quadratic loss specified by a $p \times p$ positive definite matrix $Q$: $l_Q(M,\hat{M})={\rm tr} ((\hat{M} - M) Q (\hat{M} - M)^{\top})$. 
Since this quadratic loss is related to the matrix quadratic loss by $l_Q(M,\hat{M})={\rm tr} (L(M,\hat{M}) Q)$, Theorem~\ref{th_oracle} is translated to an oracle inequality under this loss as follows.

\begin{corollary}\label{cor_oracle}
	The $Q$-quadratic risk of the Efron--Morris estimator $\hat{M}_{{\rm EM}}$ satisfies
	\begin{align}
		& {\rm E}_M [{\rm tr} ((\hat{M}_{{\rm EM}} - M) Q (\hat{M}_{{\rm EM}} - M)^{\top})] \nonumber \\
		\leq& {\rm E}_M [{\rm tr} ((\hat{M}_{C_*} - M) Q (\hat{M}_{C_*} - M)^{\top})] + 2 (p+1) {\rm tr} (Q). \label{em_oracle2}
	\end{align}
\end{corollary}

We conjecture that the constant term $2(p+1)I_p$ in \eqref{em_oracle} can be improved to $(p+1)I_p$, although the current form is enough for the present purpose.
When $p=1$, this conjecture is true from \eqref{js_oracle}, which is proved by using the Poisson mixture representation of the noncentral chi-square distribution \cite[Proposition~2.8]{Johnstone}.
The latter half of Proposition~\ref{prop_em} also supports this conjecture when $M$ has low rank.
Figure~\ref{EMrisk2} compares the Frobenius risk of the Efron--Morris estimator with the trace of the right hand side of \eqref{em_oracle} and its conjectured modification.
It implies that the improved oracle inequality holds tightly when $n$ is large, which is consistent with the finding in \cite[Figure~2.6]{Johnstone} for $p=1$.
It is an interesting future work to prove this conjecture.

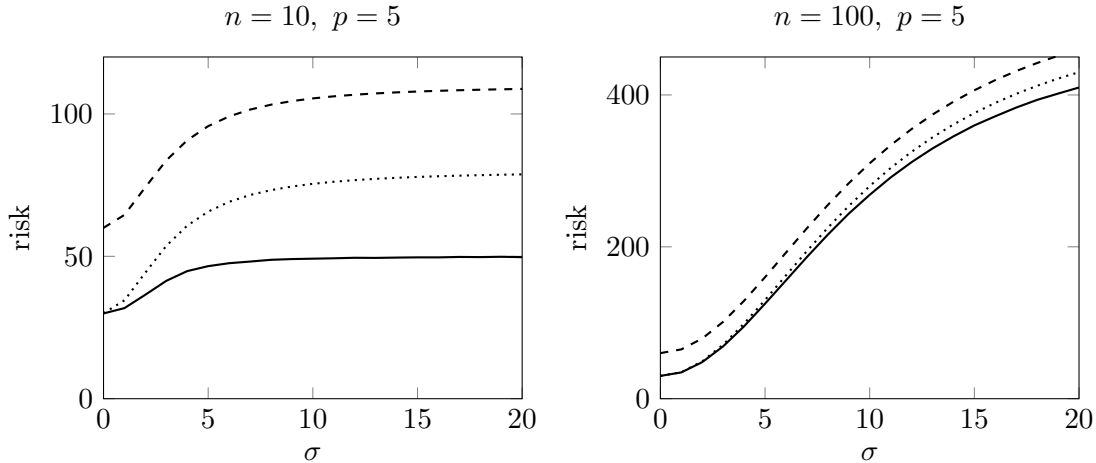
\begin{figure}[t]
	\centering
	\begin{tikzpicture}
		\begin{axis}[
			title={$n=10, \ p=5$},
			xlabel={$\sigma$}, xmin=0, xmax=20,
			ylabel={risk}, ymin=0, ymax=120,
			width=0.45\linewidth
			]
			\addplot[thick, color=black,
			filter discard warning=false, unbounded coords=discard
			] table {
				0   29.9426
				1.0000   31.8117
				2.0000   36.5250
				3.0000   41.3936
				4.0000   44.8105
				5.0000   46.5483
				6.0000   47.5839
				7.0000   48.1588
				8.0000   48.7757
				9.0000   49.0366
				10.0000   49.1683
				11.0000   49.3107
				12.0000   49.4866
				13.0000   49.4419
				14.0000   49.5548
				15.0000   49.6319
				16.0000   49.6290
				17.0000   49.7821
				18.0000   49.7258
				19.0000   49.8418
				20.0000   49.7081
				
			};
			\addplot[dashed, thick, color=black,
			filter discard warning=false, unbounded coords=discard
			] table {
				0   60.0000
				1.0000   64.5455
				2.0000   74.2857
				3.0000   83.6842
				4.0000   90.7692
				5.0000   95.7143
				6.0000   99.1304
				7.0000  101.5254
				8.0000  103.2432
				9.0000  104.5055
				10.0000  105.4545
				11.0000  106.1832
				12.0000  106.7532
				13.0000  107.2067
				14.0000  107.5728
				15.0000  107.8723
				16.0000  108.1203
				17.0000  108.3278
				18.0000  108.5030
				19.0000  108.6523
				20.0000  108.7805
			};
			\addplot[dotted, thick, color=black,
			filter discard warning=false, unbounded coords=discard
			] table {
				
				0   30.0000
				1.0000   34.5455
				2.0000   44.2857
				3.0000   53.6842
				4.0000   60.7692
				5.0000   65.7143
				6.0000   69.1304
				7.0000   71.5254
				8.0000   73.2432
				9.0000   74.5055
				10.0000   75.4545
				11.0000   76.1832
				12.0000   76.7532
				13.0000   77.2067
				14.0000   77.5728
				15.0000   77.8723
				16.0000   78.1203
				17.0000   78.3278
				18.0000   78.5030
				19.0000   78.6523
				20.0000   78.7805
			};
		\end{axis}
	\end{tikzpicture} 
	\begin{tikzpicture}
		\begin{axis}[
			title={$n=100, \ p=5$},
			xlabel={$\sigma$}, xmin=0, xmax=20,
			ylabel={risk}, ymin=0, ymax=450,
			width=0.45\linewidth
			]
			\addplot[thick, color=black,
			filter discard warning=false, unbounded coords=discard
			] table {
				0   30.0030
				1.0000   34.6215
				2.0000   48.1779
				3.0000   68.9516
				4.0000   95.3752
				5.0000  125.1165
				6.0000  155.7865
				7.0000  186.5346
				8.0000  216.0449
				9.0000  243.6997
				10.0000  268.4872
				11.0000  291.2999
				12.0000  311.4624
				13.0000  329.5191
				14.0000  345.4037
				15.0000  359.8790
				16.0000  371.9150
				17.0000  383.3799
				18.0000  393.5132
				19.0000  401.7922
				20.0000  409.7988
			};
			\addplot[dashed, thick, color=black,
			filter discard warning=false, unbounded coords=discard
			] table {
				0   60.0000
				1.0000   64.9505
				2.0000   79.2308
				3.0000  101.2844
				4.0000  128.9655
				5.0000  160.0000
				6.0000  192.3529
				7.0000  224.4295
				8.0000  255.1220
				9.0000  283.7569
				10.0000  310.0000
				11.0000  333.7557
				12.0000  355.0820
				13.0000  374.1264
				14.0000  391.0811
				15.0000  406.1538
				16.0000  419.5506
				17.0000  431.4653
				18.0000  442.0755
				19.0000  451.5401
				20.0000  460.0000
			};
			\addplot[dotted, thick, color=black,
			filter discard warning=false, unbounded coords=discard
			] table {
				0   30.0000
				1.0000   34.9505
				2.0000   49.2308
				3.0000   71.2844
				4.0000   98.9655
				5.0000  130.0000
				6.0000  162.3529
				7.0000  194.4295
				8.0000  225.1220
				9.0000  253.7569
				10.0000  280.0000
				11.0000  303.7557
				12.0000  325.0820
				13.0000  344.1264
				14.0000  361.0811
				15.0000  376.1538
				16.0000  389.5506
				17.0000  401.4653
				18.0000  412.0755
				19.0000  421.5401
				20.0000  430.0000
			};
		\end{axis}
	\end{tikzpicture} 
	\caption{Frobenius risk of the Efron--Morris estimator (solid), trace of the right hand side of \eqref{em_oracle} (dashed) and its conjectured modification (dotted) when $\sigma_1(M)=\dots=\sigma_p(M)=\sigma$. Left: $n=10$, $p=5$. Right: $n=100$, $p=5$}
	\label{EMrisk2}
\end{figure}

\section{Multivariate Gaussian sequence model and Sobolev ellipsoid}\label{sec_seq}
\subsection{Multivariate Gaussian sequence model}
For $p>1$, we consider the multivariate version of the Gaussian sequence model defined by
\begin{align}
	y_i = \theta_i + \varepsilon \xi_i, \quad i=1,2,\cdots, \label{mGS}
\end{align}
where $\varepsilon>0$ and $\xi_{i} \sim {\rm N}_p (0, I_p)$ are independent Gaussian random vectors.
This model is equivalent to the multivariate version of the Gaussian white noise model
\begin{align*}
	{\rm d} Y_j (t) = f_j (t) {\rm d} t + \varepsilon {\rm d} W_j (t),\quad j=1,\cdots,p,
\end{align*}
where $t \in [0,1]$, $f_1(t),\cdots,f_p(t)$ are drift functions to be estimated and $W_1(t),\cdots,W_p(t)$ are independent standard Wiener processes.
Specifically, for an orthonormal basis $\phi=(\phi_i)$ of $L_2 [0,1]$, let
\begin{align*}
	y_{ij} = \int \phi_i(t) {\rm d} Y_j(t), \quad \theta_{ij} = \int \phi_i(t) f_j (t) {\rm d} t, \quad \xi_{ij} = \int \phi_i(t) {\rm d} W_j(t).
\end{align*}
Then, by putting $y_i=(y_{i1},\dots,y_{ip})$, $\theta_i=(\theta_{i1},\dots,\theta_{ip})$ and $\xi_i=(\xi_{i1},\dots,\xi_{ip})$, the multivariate Gaussian sequence model \eqref{mGS} is obtained.
Also, from a similar argument to \cite{Brown}, nonparametric regression with vector response
\begin{align*}
	y_i = f(t_i) + \xi_i, \quad i=1,\dots,n,
\end{align*}
where $f:[0,1] \to \mathbb{R}^p$, $t_i=i/n$ and $\xi_i \sim {\rm N}_p (0,I_p)$ are independent, is asymptotically equivalent to the multivariate Gaussian sequence model \eqref{mGS} with $\varepsilon = n^{-1/2}$ as $n \to \infty$.

In the following, we consider estimation of $\theta=(\theta_{i})$ from $y=(y_{i})$ under the quadratic loss specified by a $p \times p$ positive definite matrix $Q$:
\begin{align}
	L_Q(\theta,\hat{\theta}) = \sum_{i=1}^{\infty} (\hat{\theta}_i-\theta_i)^{\top} Q (\hat{\theta}_i - \theta_i). \label{lossQ}
\end{align}
In the multivariate Gaussian white noise model setting, this quadratic loss corresponds to the weighted $L^2$ loss given by
\begin{align*}
	L_Q(\theta,\hat{\theta}) = \int_0^1 (\hat{Y}(t)-Y(t))^{\top} Q (\hat{Y}(t)-Y(t)) {\rm d} t.
\end{align*}

The current setting is motivated from practical applications of nonparameteric estimation.
For example, consider brain time series data such as EEG, which is a multichannel recording of brain activity.
Due to the multidimensionality, denoising of such data corresponds to nonparametric regression with vector response.
In brain-machine interface, some spatial filters are commonly applied to EEG in order to obtain relevant features for decoding brain activity \cite{Blankertz}.
Namely, a linear transformation $g(t) = A f(t)$ of EEG $f(t)$ with some matrix $A$ is considered for later use.
Since $\| \hat{g}(t)-g(t) \|^2=(\hat{f}(t)-f(t))^{\top} Q (\hat{f}(t)-f(t))$ with $Q = A^{\top}A$, mean squared error of an estimate of $g(t)$ corresponds to the quadratic loss with $Q$ of an estimate of $f(t)$.

The multivariate Gaussian sequence model is closely related to the matrix denoising model, in which a low-rank matrix is observed with noise.
For the matrix denoising model, several methods for estimating the signal matrix have been developed and their theoretical properties have been investigated.
Among them, Gavish and Donoho \cite{Gavish14} and Donoho and Gavish \cite{Donoho14} studied singular value thresholding estimators in a high-dimensional asymptotic framework.
They derived the optimal thresholds in closed form, by which the estimator can adapt to unknown rank and unknown noise level optimally in terms of mean squared error.
Later, Gavish and Donoho \cite{Gavish17} also derived the optimal singular value shrinkage rules (nonlinearities) for various loss functions.
Recently, these results have been extended to correlated noise \cite{Donoho23,Gavish23}.
Other studies on the matrix denoising model include data-driven rank selection in the context of PCA \cite{Choi} and  nuclear norm penalization \cite{Chen,Koltchinskii}.
The multivariate Gaussian sequence model can be viewed as an infinite-dimensional version of the matrix denosing model.
Thus, some generalization of the above results for the matrix denoising model may be possible, which is an interesting future problem.

\subsection{Multivariate Sobolev ellipsoid}
For $\beta>0$ and a $p \times p$ positive definite matrix $R$, we define the multivariate Sobolev ellipsoid as
\begin{align}
	\Theta (\beta,R) := \left\{ \theta = (\theta_1,\theta_2,\cdots) \relmiddle| \sum_{i=1}^{\infty} a_{\beta,i}^2 \theta_i^{\top} R^{-1} \theta_i \leq 1 \right\}, \label{ellipsoid2}
\end{align}
where $a_{\beta,i}$ is given by \eqref{sobolev_coeff}.
When $p=1$, it reduces to the original Sobolev ellipsoid \eqref{ellipsoid}.

The multivariate Sobolev ellipsoid is related to a multivariate version of the periodic Sobolev class.
For a vector-valued function $f: [0,1] \to \mathbb{R}^p$, we write each component of $f$ as $f_j: [0,1] \to \mathbb{R}$ for $j=1,\dots,p$.
For $\beta \in \{ 1,2,\cdots \}$, we define $f^{(\beta)}: [0,1] \to \mathbb{R}^p$ by $f^{(\beta)} (x)=(f_1^{(\beta)} (x),\dots,f_p^{(\beta)} (x))^{\top}$ where $f_j^{(\beta)}$ is the $\beta$-th derivative of $f_j$.
For $\beta \in \{ 1,2,\cdots \}$ and a $p \times p$ positive semidefinite matrix $L$, we define the multivariate periodic Sobolev class as
\begin{align*}
	W (\beta,L) &:=  \left\{ f: [0,1] \to \mathbb{R}^p \mid f_j \in L_2[0,1], \ f_j^{(l)} (0) = f_j^{(l)} (1), \ j=1,\dots,p, \ l=0,1,\cdots,\beta-1, \right. \\
	& \left. f_1^{(\beta - 1)}, \cdots, f_p^{(\beta - 1)} : {\rm absolutely\ continuous}, \ \int_0^1 f^{(\beta)} (x)^{\top} L^{-2} f^{(\beta)} (x) {\rm d} x \leq 1  \right\}.
\end{align*}
Then, Lemma \ref{lem_sobolev} is extended as follows.
Its proof is deferred to the Appendix.	

\begin{proposition}\label{prop_sobolev}
	For $\theta=(\theta_{ij})$, let 
	\begin{align}\label{fourier_expansion}
		f_j (x) = \sum_{i=1}^{\infty} \theta_{ij} \phi_i(x), \quad j=1,\dots,p,
\end{align}
	where $\phi = (\phi_i)$ is the Fourier basis of $L_2 [0,1]$ in \eqref{fourier}.
	Then, $f \in W (\beta,L)$ if and only if $\theta \in \Theta \left( \beta, R \right)$ in \eqref{ellipsoid2} where $R=L^2/\pi^{2 \beta}$.
\end{proposition}

Figure~\ref{sobolev_sample} plots a function in $W(\beta,L)$ with $p=2$ for several settings of $\beta$ and $L$.
Like the usual Sobolev space, the parameter $\beta$ represents the smoothness of $f$. 
On the other hand, the matrix parameter $L$ specifies not only the scale but also the correlations between $f_1,\dots,f_p$.
The functions in Figure~\ref{sobolev_sample} were obtained by using a similar method to Figure~6.3 of \cite{Johnstone}.
Specifically, we sampled $\theta_i \sim {\rm N}_p(0, 12 \pi^{-2} p^{-1} b_{\beta,i}^2 R)$ independently for $i=1,\dots,2N+1$ and substituted them into \eqref{fourier_expansion}, where $N=10^4$, $R=L^2/\pi^{2 \beta}$, $b_{\beta,1}=0$ and
\begin{align*}
	b_{\beta,i} = \begin{cases} i^{-\beta-1}, & i : {\rm even}, \\ (i-1)^{-\beta-1}, & i : {\rm odd}, \end{cases}
\end{align*}
for $i \geq 2$.
From Lemma~\ref{lem_concentration},
\begin{align*}
	\sum_{i=1}^{2N+1} a_{\beta,i}^2 \theta_i^{\top} R^{-1} \theta_i \approx \sum_{i=1}^{2N+1} a_{\beta,i}^2 {\rm tr} (R^{-1} \cdot 12 \pi^{-2} p^{-1} b_{\beta,i}^2 R) = \frac{12}{\pi^2} \sum_{m=1}^{N} \frac{2}{(2m)^2} \approx 1
\end{align*}
as $N \to \infty$.
Thus, the sampled $\theta$ is considered to lie around the boundary of $\Theta(\beta,R)$.

\begin{figure}[h]
	\centering
	\begin{tikzpicture}
		\begin{axis}[
			title={$\beta=0.5, \ L = \begin{pmatrix} 1 & 0 \\ 0 & 1 \end{pmatrix}$},
			xlabel={$x$}, xmin=0, xmax=1, ylabel={$f$},
			width=0.45\linewidth
			]
			\addplot[ thick, color=black,
			filter discard warning=false, unbounded coords=discard
			] table [x=x, y=f1] {sobolev_sample1.dat};
			\addplot[dashed,  thick, color=black,
			filter discard warning=false, unbounded coords=discard
			] table [x=x, y=f2] {sobolev_sample1.dat};
		\end{axis}
	\end{tikzpicture} 
	\begin{tikzpicture}
		\begin{axis}[
			title={$\beta=1, \ L = \begin{pmatrix} 5 & 0 \\ 0 & 1 \end{pmatrix}$},
			xlabel={$x$}, xmin=0, xmax=1, ylabel={$f$},
			width=0.45\linewidth
			]
			\addplot[ thick, color=black,
			filter discard warning=false, unbounded coords=discard
			] table [x=x, y=f1] {sobolev_sample2.dat};
			\addplot[dashed,  thick, color=black,
			filter discard warning=false, unbounded coords=discard
			] table [x=x, y=f2] {sobolev_sample2.dat};
		\end{axis}
	\end{tikzpicture} 
	\begin{tikzpicture}
		\begin{axis}[
			title={$\beta=0.5, \ L = \begin{pmatrix} 1 & 0.7 \\ 0.7 & 1 \end{pmatrix}$},
			xlabel={$x$}, xmin=0, xmax=1, ylabel={$f$},
			width=0.45\linewidth
			]
			\addplot[thick, color=black,
			filter discard warning=false, unbounded coords=discard
			] table [x=x, y=f1] {sobolev_sample3.dat};
			\addplot[dashed, thick, color=black,
			filter discard warning=false, unbounded coords=discard
			] table [x=x, y=f2] {sobolev_sample3.dat};
		\end{axis}
	\end{tikzpicture} 
	\begin{tikzpicture}
		\begin{axis}[
			title={$\beta=1, \ L = \begin{pmatrix} 1 & -0.7 \\ -0.7 & 1 \end{pmatrix}$},
			xlabel={$x$}, xmin=0, xmax=1, ylabel={$f$},
			width=0.45\linewidth
			]
			\addplot[thick, color=black,
			filter discard warning=false, unbounded coords=discard
			] table [x=x, y=f1] {sobolev_sample4.dat};
			\addplot[dashed, thick, color=black,
			filter discard warning=false, unbounded coords=discard
			] table [x=x, y=f2] {sobolev_sample4.dat};
		\end{axis}
	\end{tikzpicture} 
	\caption{Sample functions from $W(\beta,L)$. solid: $f_1$, dashed: $f_2$}
	\label{sobolev_sample}
\end{figure}
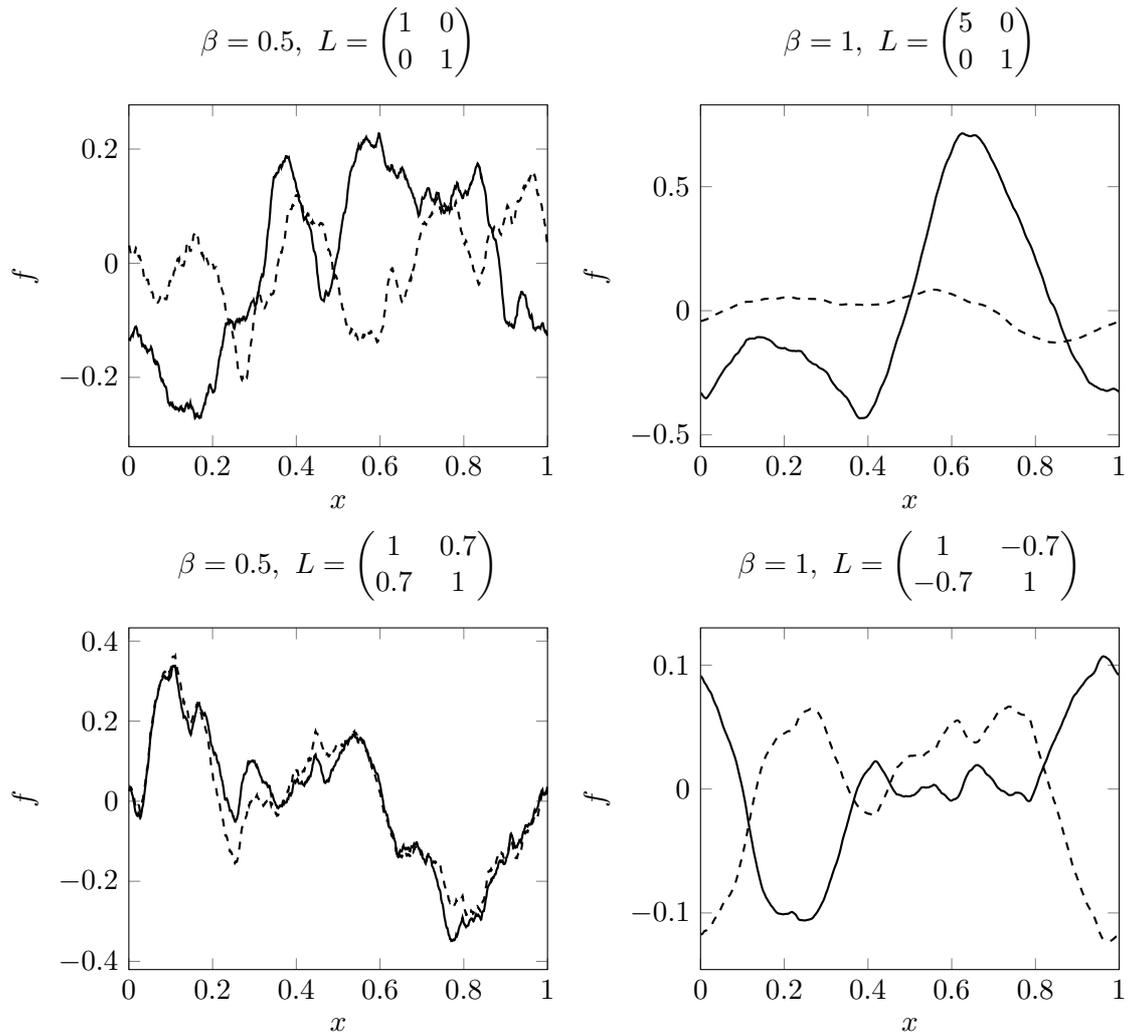

\section{Asymptotically minimax estimator on multivariate Sobolev ellipsoid}\label{sec_pinsker}
Now, we derive an asymptotically minimax estimator on the multivariate Sobolev ellipsoid \eqref{ellipsoid2}, which can be viewed as an extension of Pinsker's theorem (Proposition~\ref{prop_pinsker}). 
Several technical lemmas are given in the Appendix.
For completeness, we also extended the finite-dimensional version of Pinsker's theorem \cite{Nussbaum99,Wasserman} to multivariate setting.
Its detail is given in the Appendix.

For a symmetric matrix $S \in \mathbb{R}^{p \times p}$ that has a spectral decomposition $S=U^{\top} \Lambda U$ with $U \in O(p)$ and $\Lambda={\rm diag} (\lambda_1,\dots,\lambda_p)$, let $S_+$ be its projection onto the positive semidefinite cone defined by $S_+=U^{\top} \Lambda_+ U \succeq O$, where $\Lambda_+={\rm diag} ((\lambda_1)_+,\dots,(\lambda_p)_+)$.

\begin{lemma}\label{lem_kappa}
	Let $a=(a_i)$ be a non-decreasing sequence such that $a_i \geq 0$ and $a_i \to \infty$ as $i \to \infty$, and $Q$ be a $p \times p$ positive definite matrix with eigenvalues $\lambda_1 \geq \dots \geq \lambda_p > 0$.  
	Then, the equation
	\begin{align}
		\varepsilon^2 \kappa^{-1} \sum_{i=1}^{\infty} a_i {\rm tr} (I_p-a_i \kappa Q^{-1})_+ = 1 \label{kappa_eq}
	\end{align}
	has a unique solution of $\kappa > 0$ for every $\varepsilon>0$, which satisfies
	\begin{align}
		\kappa = \varepsilon^2 \left( 1+\varepsilon^2 \sum_{j=1}^p \lambda_j^{-1} \sum_{i=1}^{N_j(\kappa)} a_i^2 \right)^{-1} \sum_{j=1}^p \sum_{i=1}^{N_j(\kappa)} a_i, \label{kappa_eq2}
	\end{align}
	where $N_j(\kappa) = \max \{ i \mid 1- a_i \kappa \lambda_j^{-1} > 0 \}$ for $j=1,\dots,p$.
\end{lemma}
\begin{proof}	
	Let
	\begin{align*}
		g(\kappa) := \varepsilon^2 \sum_{i=1}^{\infty} a_i {\rm tr} (I_p-a_i \kappa Q^{-1})_+ = \varepsilon^2 \sum_{j=1}^p \sum_{i=1}^{N_j(\kappa)} a_i (1-a_i \kappa \lambda_j^{-1}).
	\end{align*}
	Clearly, $g$ is continuous and piecewise linear.
	For $\kappa_1 < \kappa_2$, we have $1-a_i \kappa_1 \lambda_j^{-1} > 1-a_i \kappa_2 \lambda_j^{-1}$ and $N_j(\kappa_1) \geq N_j(\kappa_2)$ for every $j$, which leads to $g(\kappa_1) > g(\kappa_2)$.
	Thus, $g(\kappa)$ is strictly monotone decreasing.
	Also, $g(\kappa) \to \infty$ as $\kappa \to 0$ and $g(\kappa) \to 0$ as $\kappa \to \infty$.
	Therefore, $\kappa^{-1} g(\kappa)$ is also strictly monotone decreasing, $\kappa^{-1} g(\kappa) \to \infty$ as $\kappa \to 0$ and $\kappa^{-1} g(\kappa) \to 0$ as $\kappa \to \infty$.
	Hence, the equation \eqref{kappa_eq}, which is equivalent to $\kappa^{-1} g(\kappa)=1$, has a unique solution of $\kappa>0$.
	From
	\begin{align*}
		\kappa^{-1} g(\kappa) = \varepsilon^2 \left( \sum_{j=1}^p \sum_{i=1}^{N_j(\kappa)} a_i \right) \kappa^{-1} - \varepsilon^2 \sum_{j=1}^p \lambda_j^{-1} \sum_{i=1}^{N_j(\kappa)} a_i^2, 
	\end{align*}
	the solution of $\kappa^{-1} g(\kappa) = 1$ satisfies \eqref{kappa_eq2}.
\end{proof}

\begin{lemma}\label{lem_kappa2}
	Let $a_i=a_{\beta,i}$ be the coefficients of the multivariate Sobolev ellipsoid \eqref{sobolev_coeff}.
	Then, the equation \eqref{kappa_eq} has a unique solution that satisfies $\kappa = \kappa_{*} (1+o(1))$ as $\varepsilon \to 0$, where
	\begin{align*}
		\kappa_{*} := \left( \frac{\beta {\rm tr} (Q^{(\beta+1)/\beta})}{(\beta+1) (2 \beta+1)} \right)^{{\beta}/{(2 \beta+1)}} \varepsilon^{{2 \beta}/{(2 \beta+1)}}.
	\end{align*}
\end{lemma}
\begin{proof}
	Since $a_{\beta,i} \geq 0$ is non-decreasing and $a_{\beta,i} \to \infty$ as $i \to \infty$, the equation \eqref{kappa_eq} has a unique solution from Lemma~\ref{lem_kappa}.
	By putting $M_j = \max \{ i \mid 1- a_{\beta,i} \kappa \lambda_j^{-1} > 0 \} = \lfloor (\kappa \lambda_j^{-1})^{-1/\beta}/2 \rfloor$, the left hand side of \eqref{kappa_eq} is rewritten as
	\begin{align*}
		&\varepsilon^2 \kappa^{-1} \sum_{j=1}^p \sum_{i=1}^{\infty} a_{\beta,i} (1-a_{\beta,i} \kappa \lambda_j^{-1})_+ = 2 \varepsilon^2 \kappa^{-1} \sum_{j=1}^p \sum_{m=1}^{M_j} (2m)^{\beta} (1-(2m)^{\beta} \kappa \lambda_j^{-1}) \\
		=& 2 \varepsilon^2 \kappa^{-1} \sum_{j=1}^p \left( \frac{2^{\beta} M_j^{\beta+1}}{\beta + 1} - \frac{4^{\beta} M_j^{2 \beta+1}}{2 \beta + 1} \kappa \lambda_j^{-1} \right) \\
		=& \frac{\beta \varepsilon^2}{(\beta+1)(2 \beta+1)} \left( \sum_{j=1}^p \lambda_j^{-1} (\kappa \lambda_j^{-1})^{-(2 \beta+1)/{\beta}} \right) (1+o(1)) \\
		=& \frac{\beta \varepsilon^2}{(\beta+1)(2 \beta+1)} {\rm tr} (Q^{(\beta+1)/{\beta}}) \kappa^{-(2 \beta+1)/{\beta}} (1+o(1)) = \left( \frac{\kappa_*}{\kappa} \right)^{(2 \beta+1)/{\beta}} (1+o(1))
	\end{align*}
	as $\varepsilon \to 0$, where we used 
	\begin{align}
		\sum_{m=1}^{M} m^a = \frac{M^{a+1}}{a+1} (1+o(1)) \label{geo_sum}
	\end{align}
	as $M \to \infty$.
	Therefore, $\kappa = \kappa_{*} (1+o(1))$ as $\varepsilon \to 0$.
\end{proof}

Consider estimation of $\theta=(\theta_i)$ in the multivariate Gaussian sequence model \eqref{mGS} under the quadratic loss $L_Q$ in \eqref{lossQ}. 
The risk function of an estimator $\hat{\theta}$ is
\begin{align*}
	R_Q(\theta,\hat{\theta}) = {\rm E}_{\theta} \left[ \sum_{i=1}^{\infty}  (\hat{\theta}_i - \theta_i)^{\top} Q (\hat{\theta}_i - \theta_i) \right].
\end{align*}
Suppose that $\theta$ is restricted to the multivariate Sobolev ellipsoid $\Theta (\beta,Q)$ in \eqref{ellipsoid2}.
From Lemma~\ref{lem_kappa2}, the equation \eqref{kappa_eq} with $a_i=a_{\beta,i}$ has a unique solution of  $\kappa>0$. 
By using this $\kappa$, we define a linear estimator $\hat{\theta}_{\mathrm{P}}=(\hat{\theta}_{\mathrm{P},i})$ by
\begin{align}
	\hat{\theta}_{\mathrm{P},i} = C_{\mathrm{P},i} y_i, \quad C_{\mathrm{P},i} = (I_p - a_{\beta,i} \kappa Q^{-1})_+,\label{mPinsker}
\end{align}
where $( \cdot )_+$ denotes the projection onto the positive semidefinite cone introduced above.
This estimator can be viewed as a generalization of Pinsker's estimator (Proposition~\ref{prop_pinsker}).

\begin{lemma}\label{lem_pinsker}
	For $C_{\mathrm{P}}=(C_{\mathrm{P},i})$ in \eqref{mPinsker},
	\begin{align}
		\varepsilon^2 \sum_{i=1}^{\infty} {\rm tr} (C_{\mathrm{P},i} Q) = P(\beta,Q) \varepsilon^{{4 \beta}/(2 \beta+1)} (1+o(1)) \label{asymp_risk}
	\end{align}
	as $\varepsilon \to 0$, where
	\begin{align}
		P(\beta,Q) := (2 \beta+1)^{1/(2 \beta+1)} \left( \frac{\beta {\rm tr} (Q^{(\beta+1)/{\beta}})}{\beta+1} \right)^{{2 \beta}/(2 \beta+1)}. \label{Pdef}
	\end{align}
\end{lemma}
\begin{proof}
	Let $\lambda_1 \geq \dots \geq \lambda_p > 0$ be the eigenvalues of $Q$ and $M_j = \max \{ i \mid 1- a_{\beta,i} \kappa \lambda_j^{-1} > 0 \} = \lfloor (\kappa \lambda_j)^{-1/\beta}/2 \rfloor$ for $j=1,\dots,p$.
	Then,
	\begin{align*}
		\varepsilon^2 \sum_{i=1}^{\infty} {\rm tr} (C_{\mathrm{P},i} Q) &= 2 \varepsilon^2 \sum_{j=1}^p \sum_{m=1}^{M_j} (1-(2m)^{\beta} \kappa \lambda_j^{-1}) \lambda_j \\
		&= 2 \varepsilon^2 \sum_{j=1}^p \left( M_j - 2^{\beta} \frac{M_j^{\beta+1}}{\beta+1} \kappa \lambda_j^{-1} \right) \lambda_j (1+o(1)) \\
		&= \varepsilon^2 \frac{\beta}{\beta+1} \kappa^{-1/\beta} {\rm tr} (Q^{(\beta+1)/{\beta}}) (1+o(1))
	\end{align*}
	as $\varepsilon \to 0$, where we used \eqref{geo_sum}.
	Therefore, from Lemma~\ref{lem_kappa2}, we obtain \eqref{asymp_risk}.
\end{proof}

\begin{theorem}\label{th_pinsker}
	The estimator $\hat{\theta}_{\mathrm{P}}$ in \eqref{mPinsker} is asymptotically minimax under $L_Q$ in \eqref{lossQ} on the multivariate Sobolev ellipsoid $\Theta(\beta,Q)$ in \eqref{ellipsoid2}:
	\begin{align*}
		\sup_{\theta \in \Theta(\beta,Q)} R_Q(\theta,\hat{\theta}_{\mathrm{P}}) &\sim \inf_{\hat{\theta}} \sup_{\theta \in \Theta(\beta,Q)} R_Q(\theta,\hat{\theta}) \sim P(\beta,Q) \varepsilon^{{4 \beta}/{(2 \beta+1)}}
	\end{align*}
	as $\varepsilon \to 0$, where $\inf$ is taken over all estimators and  $P(\beta,Q)$ is given by \eqref{Pdef}.
\end{theorem}
\begin{proof}
	We write $a_i=a_{\beta,i}$ for convenience.
	Let $Q=U^{\top} \Lambda U$ be a spectral decomposition of $Q$ with $\Lambda={\rm diag}(\lambda_1,\dots,\lambda_p)$. 
	Then, $C_{\mathrm{P},i} = U^{\top} (I_p-a_i \kappa \Lambda^{-1})_+ U$.
	Therefore, by using $(1-(1-l)_+)^2 \leq l^2$,
	\begin{align}
		(I_p-C_{\mathrm{P},i})Q(I_p-C_{\mathrm{P},i}) &= U^{\top} (I_p-(I_p-a_i \kappa \Lambda^{-1})_+)^2 \Lambda U \nonumber \\
		&\preceq U^{\top} \cdot a_i^2 \kappa^2 \Lambda^{-2} \cdot \Lambda \cdot U =a_i^2 \kappa^2 Q^{-1}. \label{mat_ineq}
	\end{align}
	Also, by using $a_i \kappa \lambda_j^{-1} + (1-a_i \kappa \lambda_j^{-1})_+=1$ when $(1-a_i \kappa \lambda_j^{-1})_+ > 0$,
	\begin{align}\label{lem_risk}
		\sum_{i=1}^{\infty} {\rm tr} (C_{\mathrm{P},i} Q (a_i \kappa Q^{-1} + C_{\mathrm{P},i})) &= \sum_{i=1}^{\infty} \sum_{j=1}^p (1-a_i \kappa \lambda_j^{-1})_+ \lambda_j (a_i \kappa \lambda_j^{-1} + (1-a_i \kappa \lambda_j^{-1})_+) \nonumber \\
		&= \sum_{i=1}^{\infty} \sum_{j=1}^p (1-a_i \kappa \lambda_j^{-1})_+ \lambda_j = \sum_{i=1}^{\infty} {\rm tr} (C_{\mathrm{P},i} Q).
	\end{align}
	Similarly,
	\begin{align}\label{lem_risk_inv}
		\sum_{i=1}^{\infty} {\rm tr} (C_{\mathrm{P},i} Q (a_i \kappa Q^{-1} + C_{\mathrm{P},i})^{-1}) &= \sum_{i=1}^{\infty} {\rm tr} (C_{\mathrm{P},i} Q).
	\end{align}
	
	Now, from
	\begin{align*}
		{\rm E}_{\theta} \left[ (\hat{\theta}_{\mathrm{P},i}-\theta_i) (\hat{\theta}_{\mathrm{P},i}-\theta_i)^{\top} \right] = (I_p-C_{\mathrm{P},i}) \theta_i \theta_i^{\top} (I_p-C_{\mathrm{P},i}) + \varepsilon^2 C_{\mathrm{P},i}^2
	\end{align*}
	and \eqref{mat_ineq}, we obtain
	\begin{align*}
		R_Q (\theta,\hat{\theta}_{\mathrm{P}}) &= \sum_{i=1}^{\infty} \theta_i^{\top} (I_p-C_{\mathrm{P},i}) Q (I_p-C_{\mathrm{P},i}) \theta_i + \varepsilon^2 \sum_{i=1}^{\infty} {\rm tr} (C_{\mathrm{P},i}^2 Q) \\
		&\leq \kappa^2 \sum_{i=1}^{\infty} a_i^2  \theta_i^{\top} Q^{-1} \theta_i + \varepsilon^2 \sum_{i=1}^{\infty} {\rm tr} (C_{\mathrm{P},i}^2 Q) \\
		&\leq \kappa^2 + \varepsilon^2 \sum_{i=1}^{\infty} {\rm tr} (C_{\mathrm{P},i}^2 Q)
	\end{align*}
	for $\theta \in \Theta(\beta,Q)$.
	By using \eqref{kappa_eq} and \eqref{lem_risk},
	\begin{align*}
		\kappa^2 + \varepsilon^2 \sum_{i=1}^{\infty} {\rm tr} (C_{\mathrm{P},i}^2 Q) &= \varepsilon^2 \kappa \sum_{i=1}^{\infty} a_i  {\rm tr} (I_p-a_i \kappa Q^{-1})_+ + \varepsilon^2 \sum_{i=1}^{\infty} {\rm tr} (C_{\mathrm{P},i}^2 Q) \\
		&= \varepsilon^2 \sum_{i=1}^{\infty} {\rm tr} (C_{\mathrm{P},i} Q (a_i \kappa Q^{-1} + C_{\mathrm{P},i})) \\
		&= \varepsilon^2 \sum_{i=1}^{\infty} {\rm tr} (C_{\mathrm{P},i} Q).
	\end{align*}
	Therefore,  
	\begin{align}
		\inf_{\hat{\theta}} \sup_{\theta \in \Theta(\beta,Q)} R_Q (\theta,\hat{\theta}) &\leq \sup_{\theta \in \Theta(\beta,Q)} R_Q (\theta,\hat{\theta}_{\mathrm{P}}) \leq \varepsilon^2 \sum_{i=1}^{\infty} {\rm tr} (C_{\mathrm{P},i} Q). \label{upper_bound}
	\end{align}
	
	Now, we consider the minimax lower bound. Let $N=N_1(\kappa)=\max \{ i \mid 1- a_i \kappa \lambda_1^{-1} > 0 \}$ and
	\begin{align*}
		\Theta_N(\beta,Q) := \left\{ \theta = (\theta_1,\theta_2,\dots) \relmiddle| \theta_1=0; \ \sum_{i=2}^{N} a_i^2 \theta_i^{\top} Q^{-1} \theta_i \leq 1; \  \theta_i = 0, i > N \right\}, \end{align*}
	where $\theta_1=0$ comes from $a_1=0$.
	Then, from $\Theta_N(\beta,Q) \subset \Theta(\beta,Q)$,
	\begin{align}
		\inf_{\hat{\theta}} \sup_{\theta \in \Theta(\beta,Q)} R_Q (\theta,\hat{\theta}) &\geq \inf_{\hat{\theta}} \sup_{\theta \in \Theta_N(\beta,Q)} R_Q (\theta,\hat{\theta}) \nonumber \\
		&= \inf_{\hat{\theta}} \sup_{\theta \in \Theta_N(\beta,Q)} {\rm E}_{\theta} \left[ \sum_{i=2}^N  (\hat{\theta}_i - \theta_i)^{\top} Q (\hat{\theta}_i - \theta_i) \right], \label{lower_bound}
	\end{align}
	where the equality follows because the infimum is attained by an estimator such that $\hat{\theta}_i=0$ for $i=1$ and $i>N$.
	Thus, we restrict our attention to estimation of $\theta^N=(\theta_2,\dots,\theta_N)$ in the following.
	Let $\delta \in (0,1)$ and $\pi$ be the prior on $\mathbb{R}^{N}$ such that $\theta_i \sim {\rm N}_p(0,\delta^2 V_i)$ independently with $V_i=\varepsilon^2 a_i^{-1}  \kappa^{-1} Q (I_p-a_i  \kappa Q^{-1})_+$ for $i=2,\dots,N$.
	From Lemma~\ref{lem_bayes_risk}, the Bayes estimator with respect to this prior (posterior mean) is
	\begin{align*}
		\hat{\theta}^{\pi}_i = C^{\pi}_i x_i, \quad C^{\pi}_i = \delta^2 V_i (\delta^2 V_i + \varepsilon^2 I_p)^{-1}, 
	\end{align*}
	for $i=2,\dots,N$, and its Bayes risk is
	\begin{align}
		r(\pi) = \delta^2 \varepsilon^2 \sum_{i=2}^{N} {\rm tr} (Q V_i (\delta^2 V_i + \varepsilon^2 I_p)^{-1}) &\geq \delta^2 \varepsilon^2 \sum_{i=2}^{N} {\rm tr} (Q V_i (V_i + \varepsilon^2 I_p)^{-1}) \nonumber \\
		&= \delta^2 \varepsilon^2 \sum_{i=1}^{\infty} {\rm tr} (C_{\mathrm{P},i} Q (a_i \kappa Q^{-1} + C_{\mathrm{P},i})^{-1}) \nonumber \\
		&= \delta^2 \varepsilon^2 \sum_{i=1}^{\infty} {\rm tr} (C_{\mathrm{P},i} Q), \label{bayes_risk}
	\end{align}
	where we used \eqref{lem_risk_inv}.
	Let 
	\begin{align*}
		\Phi := \left\{ \theta^N = (\theta_2,\dots,\theta_N) \relmiddle| \sum_{i=2}^{N} a_i^2 \theta_i^{\top} Q^{-1} \theta_i \leq 1 \right\} \subset \mathbb{R}^{p(N-1)}
	\end{align*}
	be the truncation of $\Theta_N(a,Q)$ and $\mathbb{E}$ be the class of estimators $\hat{\theta}^N$ of ${\theta}^N = ({\theta}_2,\dots,{\theta}_N)$ that take values in $\Phi$.
	Since $\Phi$ is closed and convex, we can restrict our attention to estimators in $\mathbb{E}$ in considering the lower bound for the minimax risk over $\Phi$.
	Let
	\begin{align*}
		R_{N,Q}(\theta^N,\hat{\theta}^N) := {\rm E}_{\theta^N} \left[ \sum_{i=2}^{N}  (\hat{\theta}_i - \theta_i)^{\top} Q (\hat{\theta}_i - \theta_i) \right].
	\end{align*}
	Then, for any estimator $\hat{\theta}^N \in \mathbb{E}$, from the definition of the Bayes risk,
	\begin{align*}
		r(\pi) &\leq \int R_{N,Q}(\theta^N,\hat{\theta}^N) \pi(\theta^N) {\rm d} \theta^N \\
		& \leq \sup_{\theta^N \in \Phi} R_{N,Q}(\theta^N,\hat{\theta}^N) + \sup_{\hat{\theta}^N \in \mathbb{E}} \int_{\Phi^c} R_{N,Q}(\theta^N,\hat{\theta}^N) \pi(\theta^N) {\rm d} \theta^N,
	\end{align*}
	where $\Phi^c=\mathbb{R}^{p(N-1)} \setminus \Phi$.
	Taking the infimum over all estimators in $\mathbb{E}$,
	\begin{align*}
		r(\pi) & \leq \inf_{\hat{\theta}^N \in \mathbb{E}} \sup_{\theta^N \in \Phi} R_{N,Q}(\theta^N,\hat{\theta}^N) + \sup_{\hat{\theta}^N \in \mathbb{E}} \int_{\Phi^c} R_{N,Q}(\theta^N,\hat{\theta}^N) \pi(\theta^N) {\rm d} \theta^N,
	\end{align*}
	which yields
	\begin{align}
		\inf_{\hat{\theta}^N \in \mathbb{E}} \sup_{\theta^N \in \Phi} R_{N,Q}(\theta^N,\hat{\theta}^N) \geq r(\pi) - \sup_{\hat{\theta}^N \in \mathbb{E}} \int_{\Phi^c} R_{N,Q}(\theta^N,\hat{\theta}^N) \pi(\theta^N) {\rm d} \theta^N. \label{lower_bound2}
	\end{align}
	
	Now, from $(a+b)^{\top} Q (a+b) \leq 2 (a^{\top} Q a + b^{\top} Q b)$ and the Cauchy-Schwarz inequality,
	\begin{align}
		&\sup_{\hat{\theta}^N \in \mathbb{E}} \int_{\Phi^c} R_{N,Q}(\theta^N,\hat{\theta}^N) \pi(\theta^N) {\rm d} \theta^N \nonumber \\
		\leq& 2 \int_{\Phi^c} \left( \sum_{i=2}^{N} \theta_i^{\top} Q \theta_i \right) \pi(\theta^N) {\rm d} \theta^N + 2 \sup_{\hat{\theta}^N \in \mathbb{E}} \int_{\Phi^c} {\rm E}_{\theta^N} \left[ \sum_{i=2}^{N} \hat{\theta}_i^{\top} Q \hat{\theta}_i  \right] \pi(\theta^N) {\rm d} \theta^N \nonumber \\
		\leq& 2 (\pi(\Phi^c))^{1/2} {\rm E}_{\pi} \left[ \left( \sum_{i=1}^N \theta_i^{\top} Q \theta_i \right)^2 \right]^{1/2} + 2 M \pi (\Phi^c), \label{bound}
	\end{align}
	where
	\begin{align*}
		M := \max_{\theta^N \in \Phi} \sum_{i=2}^N {\theta}_i^{\top} Q {\theta}_i = {\lambda_1^2}{a_2^{-2}} = O(1).
	\end{align*}
	Let $\xi_i = \delta^{-1} V_i^{-1/2} \theta_i \sim {\rm N}_p(0,I_p)$ and $V_i^{1/2} Q^{-1} V_i^{1/2} = U^{\top} \Lambda_i U$ with $\Lambda_i={\rm diag}(\lambda_{i1},\dots,\lambda_{ip})$ be a spectral decomposition of $V_i^{1/2} Q^{-1} V_i^{1/2}$.
	Then,
	\begin{align*}
		\theta_i^{\top} Q^{-1} \theta_i = \delta^2 \xi_i^{\top} V_i^{1/2} Q^{-1} V_i^{1/2} \xi_i = \delta^2 (U \xi_i)^{\top} \Lambda_i (U \xi_i) = \delta^2 \sum_{j=1}^p \lambda_{ij} (U {\xi}_i)_j^2.
	\end{align*}
	Therefore, by putting $b_{ij}^2 = \delta^2 a_i^2 \lambda_{ij} = \delta^2 \varepsilon^2 \kappa^{-1} a_i (1-a_i \kappa \lambda_j^{-1})_+$,
	\begin{align*}
		\pi (\Phi^c) & = {\rm Pr} \left[ \sum_{i=2}^N a_i^2 \theta_i^{\top} Q^{-1} \theta_i > 1 \right] \nonumber \\
		&= {\rm Pr} \left[ \sum_{i=2}^N \sum_{j=1}^p b_{ij}^2 ((U\xi_i)_j^2-1) > \frac{1-\delta^2}{\delta^2} \sum_{i=2}^N \sum_{j=1}^p b_{ij}^2 \right],
	\end{align*}
	where we used $\sum_{i,j} b_{ij}^2 = \delta^2$.
	Since $U \xi_i \sim {\rm N}_p (0,I_p)$ independently for $i=1,\dots,N$, from Lemma~\ref{lem_chi2},
	\begin{align}
		\pi (\Phi^c) &= O \left( \exp \left( - \frac{(1-\delta^2)^2}{8 \delta^4} \frac{\sum_{i,j} b_{ij}^2}{\max_{i,j} \{ b_{ij}^2 \mid b_{ij}^2>0 \} } \right) \right) \nonumber \\
		&= O \left( \exp \left( - C \varepsilon^{-2/(2\beta+1)} \right) \right) \label{tail_prob}
	\end{align}
	for $\delta \in (1/\sqrt{2},1)$ with some constant $C>0$, where we used $b_{ij}^2 \leq \delta^2 \varepsilon^2 \kappa^{-1} a_{N_j(\kappa)} \leq \delta^2 \varepsilon^2 \kappa^{-2} \lambda_j = O(\varepsilon^{2/(2\beta+1)})$ for $i \leq N_j(\kappa)$ and $b_{ij}^2 = 0$ for $i > N_j(\kappa)$.
	Also, by putting $V_i^{1/2} Q V_i^{1/2} = U^{\top} D U$ with $D= {\rm diag}(\mu_{i1},\dots,\mu_{ip})$,
	\begin{align*}
		\theta_i^{\top} Q \theta_i = \delta^2 \xi_i^{\top} V_i^{1/2} Q V_i^{1/2} \xi_i = \delta^2 (U \xi_i)^{\top} D (U \xi_i) = \delta^2 \sum_{j=1}^p \mu_{ij} (U \xi_i)_j^2.
	\end{align*}
	Thus, from Lemma~\ref{lem_concentration},
	\begin{align}
		{\rm E}_{\pi} \left[ \left( \sum_{i=2}^N \theta_i^{\top} Q \theta_i \right)^2 \right] &\leq 3 \delta^4  \left( \sum_{i=2}^N \sum_{j=1}^p \mu_{ij} \right)^2 \nonumber \\
		&= 3 \delta^4 \left( \sum_{i=2}^N {\rm tr} (Q V_i) \right)^2 \nonumber \\
		&\leq 3 \delta^4 a_2^{-4} \left( \sum_{i=2}^N a_i^2 {\rm tr} (Q V_i) \right)^2 \nonumber \\
		&= 3 \delta^4 a_2^{-4} \left( \varepsilon^2 \kappa^{-1} \sum_{i=2}^N a_i {\rm tr} (C_{\mathrm{P},i} Q^2) \right)^2 \nonumber \\
		&\leq 3 \delta^4 \lambda_1^4 a_2^{-4} \left( \varepsilon^2 \kappa^{-1} \sum_{i=2}^N a_i {\rm tr} (C_{\mathrm{P},i}) \right)^2 \nonumber \\
		&= 3 \delta^4 \lambda_1^4 a_2^{-4} \nonumber \\
		&= O(1). \label{finite_exp}
	\end{align}
	Substituting \eqref{tail_prob} and \eqref{finite_exp} into \eqref{bound} yields
	\begin{align*}
		&\sup_{\hat{\theta}^N \in \mathbb{E}} \int_{\Phi^c} R(\theta^N,\hat{\theta}^N) \pi(\theta^N) {\rm d} \theta^N = O \left( \exp \left( - C \varepsilon^{-2/(2\beta+1)} \right) \right).
	\end{align*}
	Combining it with \eqref{lower_bound}, \eqref{bayes_risk} and \eqref{lower_bound2} and using Lemma~\ref{lem_pinsker}, we obtain
	\begin{align*}
		\inf_{\hat{\theta}} \sup_{\theta \in \Theta(\beta,Q)} R_Q(\theta,\hat{\theta}) \geq \delta^2 \varepsilon^2 \sum_{i=1}^{\infty} {\rm tr} (C_{\mathrm{P},i} Q) (1+o(1)).
	\end{align*}
	By taking $\delta \to 1$, 
	\begin{align}
		\inf_{\hat{\theta}} \sup_{\theta \in \Theta(\beta,Q)} R_Q(\theta,\hat{\theta}) &\geq \varepsilon^2 \sum_{i=1}^{\infty} {\rm tr} (C_{\mathrm{P},i} Q) (1+o(1)). \label{minimax_lower_bound}
	\end{align}
	
	From \eqref{upper_bound}, \eqref{minimax_lower_bound} and Lemma~\ref{lem_pinsker}, we obtain the theorem.
\end{proof}

Theorem~\ref{th_pinsker} implies that the multivariate Sobolev ellipsoid $\Theta(\beta,c Q)$ with $\beta>0$ and $c>0$ is a canonical parameter space when the quadratic loss $L_Q$ is adopted, in the sense that a linear estimator attains the asymptotic minimaxity.
Note that the (asymptotic) minimaxity under $L_Q$ is equivalent to the (asymptotic) minimaxity under $L_{aQ}$ for every $a>0$ since $L_{aQ}=a L_Q$.
For the parameter space $\sum_{i=1}^{\infty} a_{\beta,i}^2 \theta_i^{\top} Q^{-1} \theta_i \leq K$, the asymptotic minimax risk for the loss $L_Q$ is given by 
\[
R^{-1} P(\beta,KQ) \varepsilon^{4\beta/(2\beta+1)} = K^{1/(2\beta+1)} P(\beta,Q) \varepsilon^{4\beta/(2\beta+1)},
\]
where we used $L_{KQ} = K L_Q$.
For example, the multichannel signal detection problem studied by \cite{Ingster,Gayraud} naturally corresponds to $K=p$.

Whereas we gave a detailed proof of Theorem~\ref{th_pinsker} for completeness, it can be viewed as a special case of a general version of Pinsker’s theorem (e.g. Theorem 5.1 in \cite{Johnstone}). 
In fact, from this viewpoint, Theorem~\ref{th_pinsker} can be extended to estimation of $\theta \in \Theta(\beta,R)$ under $L_Q$ where $Q$ and $R$ are simultaneously diagonalizable (thus commutable) as follows\footnote{We thank the referee for pointing out this.}.
Let $R=U^{\top} \Lambda U$ be a spectral decomposition of $R$.
Since $Q$ and $R$ are simultaneously diagonalizable, we have $Q=U^{\top} D U$ with a diagonal matrix $D$.
Let $\xi_i := D^{1/2} U \theta_i$.
Since 
\begin{align*}
	\sum_{i=1}^{\infty} a_i^2 \theta_i^{\top} R^{-1} \theta_i = \sum_{i=1}^{\infty} a_i^2 \xi_i^{\top} (\Lambda D)^{-1} \xi_i, 
\end{align*}
we have $\theta \in \Theta(\beta,R)$ if and only if $\xi \in \Theta(\beta,\Lambda D)$.
Also, $L_Q(\theta,\hat{\theta})=\sum_i \| \hat{\xi}_i-\xi_i \|^2$.
Rewrite
\begin{align*}
	\sum_{i=1}^{\infty} a_i^2 \xi_i^{\top} (\Lambda D)^{-1} \xi_i = \sum_{k=1}^{\infty} b_k^2 s_k^2,
\end{align*}
where $b=(b_k)$ is a non-decreasing reordering of 
\begin{align*}
	a_1^2 ((\Lambda D)^{-1})_{11},\dots,a_1^2 ((\Lambda D)^{-1})_{pp},a_2^2 ((\Lambda D)^{-1})_{11},\dots,a_2^2 ((\Lambda D)^{-1})_{pp},\dots,
\end{align*}
and $s=(s_k)$ is the same reordering of $\xi_{11},\dots,\xi_{1p}, \xi_{21},\dots,\xi_{2p},\dots$.
Let $z_k$ be the same reordering of 
\begin{align*}
	(D^{1/2} U x_1)_1,\dots,(D^{1/2} U x_1)_p,(D^{1/2} U x_2)_1,\dots,(D^{1/2} U x_2)_p,\dots.
\end{align*}
Then, $z_k \sim {\rm N}(s_k,\varepsilon_k^2)$ where $\varepsilon_k^2$ is either of $D_{11},\dots,D_{pp}$.
In this way, the problem is reduced to estimation of the mean sequence under the usual quadratic loss in the original Gaussian sequence model ($p=1$) with heteroscedastic variance and an ellipsoid constraint.
Therefore, from Pinsker's theorem (e.g. Theorem 5.1 in \cite{Johnstone}), a linear estimator attains asymptotic minimaxity.

From the proof of Theorem~\ref{th_pinsker}, the independent Gaussian prior $\theta_i \sim {\rm N}_p (0,\delta^2 V_i)$ with $V_i=\varepsilon^2 a_i^{-1}  \kappa^{-1} Q (I_p-a_i  \kappa Q^{-1})_+$ and $\delta \approx 1$ is considered to be nearly least favorable in that its Bayes risk is comparable to the minimax risk \cite{Johnstone}. 
Intuitively, the probability mass of this prior concentrates around the boundary of $\Theta(\beta,Q)$ as $\varepsilon \to 0$.
We can visualize this least favorable prior like Figure~6.3 of \cite{Johnstone} (see Figure~\ref{sobolev_sample}).

It is an interesting problem whether Theorem~\ref{th_pinsker} can be extended to general $L_Q$ and $\Theta(\beta,R)$ with $Q \neq R$.
To discuss it, let us consider a simpler problem of estimating $\theta \in \mathbb{R}^d$ restricted to $\Theta(R)=\{ \theta \mid \theta^{\top} R^{-1} \theta = 1 \}$ from the observation $X \sim {\rm N}_d(\theta,I_d)$ under the $Q$-quadratic loss $L_Q(\theta,\hat{\theta})=(\hat{\theta}-\theta)^{\top} Q (\hat{\theta}-\theta)$, where $Q$ and $R$ are $d \times d$ positive definite matrices.
This problem has been well studied in the case of $Q=I_d$ and $R=r^2 I_d$, which is closely related to Pinsker's theorem \cite{Beran}. 
In this case, $\Theta(R)$ is the sphere of radius $r$ defined by $S^{d-1}(r)=\{ \theta \mid \| \theta \| = r \}$ and the Bayes estimator with respect to the uniform prior on $S^{d-1}(r)$ is exactly minimax.
Also, the Bayes estimator with respect to the Gaussian prior ${\rm N}_d(0,(r^2/d) I_d)$, which is the linear estimator $\hat{\theta}(x)=r^2/(d+r^2) \cdot x$, attains asymptotic minimaxity as $d \to \infty$ \cite{Beran}.
Intuitively, it is understood as the Gaussian prior being closer to the uniform prior on $S^{d-1}(r)$ as $d \to \infty$ by the concentraion of measure.
An important point here is that the risk functions of both estimators are constant on $S^{d-1}(r)$ due to the orthogonal invariance\footnote{If the parameter space is set to the ball $B^{d-1}(r)=\{ \theta \mid \| \theta \| \leq r \}$, then the risk functions of both estimators are maximized on $S^{d-1}(r)$: $\argmax_{\theta \in B^{d-1}(r)} R(\theta,\hat{\theta}) = S^{d-1}(r)$.}.
Note that a Bayes estimator is (exactly) minimax if and only if all maximizers of its risk function belong to the support of the prior in the current setting \cite{Lehmann}.
However, for general $Q$ and $R$, the risk functions of both estimators are not constant on $S^{d-1}(r)$.
Thus, it is not clear whether a linear estimator attains asymptotic minimaxity in such general cases\footnote{In fact, this problem was mentioned in Exercise 5.3 of \cite{Johnstone} as ``The situation appears to be less simple if A and $\Sigma$ do not commute."}.
Investigation of this problem would clarify whether Theorem~\ref{th_pinsker} can be extended to general $Q$ and $R$.
We leave it for future work.

\section{Adaptive minimaxity of blockwise Efron--Morris estimator}\label{sec_bem}
\subsection{Monotone and blockwise constant oracles}
For the multivariate Gaussian sequence model \eqref{mGS}, the linear estimator with coefficient matrices $C=(C_i)$ where $C_i \in \mathbb{R}^{p \times p}$ is defined by
\begin{align*}
	\hat{\theta}_C=(\hat{\theta}_{C,i}), \quad \hat{\theta}_{C,i} = C_i y_i.
\end{align*}
Here, we consider two classes of coefficient matrices that depend on the choice of the quadratic loss $L_Q$.
Let $Q = U^{\top} \Lambda U$ be a spectral decomposition of $Q$ and $N$ be a positive integer, which will be set to $\lfloor \varepsilon^{-2} \rfloor$ in the next subsection.

The first one is the class of monotone coefficient matrices commutative with $Q$:
\begin{align*}
	\mathbb{W}_{\mathrm{mon},Q} := \{ C=(C_i) \mid C_i &= U^{\top} \Lambda_i U, \ \Lambda_i: \mathrm{diagonal}, \ i=1,\dots,N; \\ 
	& I_p \succeq \Lambda_1 \succeq \Lambda_2 \succeq \dots \succeq \Lambda_{N} \succeq O; \ C_i=O, \ i > N \}.
\end{align*}
The second one is the class of blockwise constant coefficient matrices  commutative with $Q$:
\begin{align*}
	\mathbb{W}_{B,Q} = \{ C=({C}_i) \mid {C}_i &={C}_{(j)}, \ i \in B_j; \ C_i=O, \ i > N; \\
	& C_{(j)} = U^{\top} \Lambda_{(j)} U, \ \Lambda_{(j)}: \mathrm{diagonal}, \ O \preceq \Lambda_{(j)} \preceq I_p \},
\end{align*}
where $B=(B_1,\dots,B_J)$ is a partition of $\{ 1,2,\dots,N\}$ that satisfies $\max B_{j-1} < \min B_j$ for $j=2,\dots,J$, similar to the weakly geometric blocks introduced in Section~\ref{sec_bjs}.

For a fixed $\theta$, we can consider two oracle estimators that attain the smallest risk among $\mathbb{W}_{\mathrm{mon},Q}$ and $\mathbb{W}_{B,Q}$, respectively.
The risks of these oracle estimators are related as follows.

\begin{lemma}\label{lem_block}
	If ${|B_{j+1}|} \leq (1+\eta) |B_j|$ for $j=1,\dots,J-1$ with $\eta>0$, then
	\begin{align}
		\min_{\bar{C} \in \mathbb{W}_{B,Q}} R_Q(\theta,\hat{\theta}_{\bar{C}}) \leq (1+\eta) \min_{C \in \mathbb{W}_{\mathrm{mon},Q}} R_Q(\theta,\hat{\theta}_C) + \varepsilon^2 |B_1| {\rm tr} (Q). \label{mon_block} 
	\end{align}
\end{lemma}
\begin{proof}
	For $C=(C_i) \in \mathbb{W}_{\mathrm{mon},Q}$, define $\bar{C}=(\bar{C}_1,\bar{C}_2,\dots) \in \mathbb{W}_{B,Q}$ by
	\begin{align*}
		\bar{C}_i := \begin{cases} \bar{C}_{(j)}, & i \in B_j, \\ O, & i > N, \end{cases}
	\end{align*}
	where $\bar{C}_{(j)}=C_{\min B_j}$.
	For convenience, we put $\Lambda_{(j)} = \Lambda_{\min B_j}$ so that $\bar{C}_{(j)}= U^{\top} \Lambda_{(j)} U$.
	Then, 
	\begin{align*}
		(I_p-\bar{C}_i) Q (I_p-\bar{C}_i) &= U^{\top} (I_p-{\Lambda}_{(j)}) \Lambda (I_p-{\Lambda}_{(j)}) U \\
		& \preceq U^{\top} (I_p-{\Lambda}_i) \Lambda (I_p-{\Lambda}_i) U \\
		&= (I_p-{C}_i) Q (I_p-{C}_i) 
	\end{align*}
	for $i \in B_j$, since $\Lambda,\Lambda_i,\Lambda_{(j)}$ are all diagonal and $O \preceq \Lambda_i \preceq \Lambda_{(j)} \preceq I_p$.
	Therefore,
	\begin{align*}
		R_Q(\theta,\hat{\theta}_{\bar{C}}) &= \sum_{i=1}^{\infty} ( \theta_i^{\top} (I_p-\bar{C}_i) Q (I_p-\bar{C}_i) \theta_i + \varepsilon^2 {\rm tr} (\bar{C}_i Q \bar{C}_i) ) \\
		& \leq \sum_{i=1}^{\infty} \theta_i^{\top} (I_p-C_i) Q (I_p-C_i) \theta_i + \varepsilon^2 {\rm tr} \left( \sum_{i=1}^{\infty} \bar{C}_i Q \bar{C}_i \right).
	\end{align*}
	By definition of $\bar{C}$,
	\begin{align*}
		\sum_{i=1}^{\infty} \bar{C}_i Q \bar{C}_i &= \sum_{j=1}^{J} |B_j| \bar{C}_{(j)} Q \bar{C}_{(j)} \preceq  |B_1| Q + \sum_{j=2}^J |B_j| \bar{C}_{(j)} Q \bar{C}_{(j)},
	\end{align*}
	where we used $\bar{C}_{(1)} Q \bar{C}_{(1)}=U^{\top} \Lambda_{(1)}^2 \Lambda U \preceq U^{\top} \Lambda U=Q$.
	From 
	\begin{align*}
		\bar{C}_{(j)} Q \bar{C}_{(j)} = U^{\top} \Lambda_{(j)}^2 \Lambda U \preceq U^{\top} \Lambda_{(j-1)}^2 \Lambda U = \bar{C}_{(j-1)} Q \bar{C}_{(j-1)}
	\end{align*}
	and $|B_j| \leq (1+\eta) |B_{j-1}|$ for $j \geq 2$,
	\begin{align*}
		\sum_{j=2}^J |B_j| \bar{C}_{(j)} Q \bar{C}_{(j)} &\preceq (1+\eta) \sum_{j=2}^J |B_{j-1}| \bar{C}_{(j-1)} Q \bar{C}_{(j-1)} \\
		& \preceq (1+\eta) \sum_{j=1}^{J-1} \sum_{i \in B_j} {C}_{i} Q C_i  \preceq (1+\eta) \sum_{i=1}^{\infty} {C}_{i} Q C_i.
	\end{align*}
	Therefore,
	\begin{align*}
		R_Q(\theta,\hat{\theta}_{\bar{C}}) &\leq \sum_{i=1}^{\infty} \theta_i^{\top} (I_p-C_i) Q (I_p-C_i) \theta_i + \varepsilon^2 {\rm tr} \left(|B_1| Q + (1+\eta) \sum_{i=1}^{\infty} {C}_{i} Q {C}_{i} \right) \\
		&\leq (1+\eta) \sum_{i=1}^{\infty} ( \theta_i^{\top} (I_p-C_i) Q (I_p-C_i) \theta_i + \varepsilon^2 {\rm tr} ({C}_i Q {C}_i) ) + \varepsilon^2 |B_1| {\rm tr} (Q) \\
		&= (1+\eta) R_Q(\theta,\hat{\theta}_{{C}}) + \varepsilon^2 |B_1| {\rm tr} (Q).
	\end{align*}
	Since $C  \in \mathbb{W}_{\mathrm{mon},Q}$ is arbitrary, \eqref{mon_block} is obtained.
\end{proof}

\subsection{Blockwise Efron--Morris estimator}
Let $N=\lfloor \varepsilon^{-2} \rfloor$ and $B=(B_1,\dots,B_J)$ be a partition of $\{ 1,\dots,N \}$.
For the multivariate Gaussian sequence model \eqref{mGS}, we define the blockwise Efron--Morris estimator $\hat{\theta}_{\mathrm{BEM}}=(\hat{\theta}_{\mathrm{BEM},i})$ by
\begin{align}
	\hat{\theta}_{\mathrm{BEM},i} = \begin{cases} y_i, & i \in B_j, \ |B_j| < p+2, \\ ( I_p-\varepsilon^2(|B_j|-p-1) (\sum_{i \in B_j} y_i y_i^{\top})^{-1} ) y_i, & i \in B_j, \ |B_j| \geq p+2, \\ 0, & i > N. \end{cases} \label{BEM}
\end{align}
This estimator applies the Efron--Morris type singular value shrinkage to each block $B_j$.
For the weakly geometric blocks introduced in Section~\ref{sec_bjs}, the blockwise Efron--Morris estimator attains the adaptive minimaxity over the multivariate Sobolev ellipsoids as follows.

\begin{theorem}\label{th_bem}
	For every $\beta$ and $Q$, the blockwise Efron--Morris estimator $\hat{\theta}_{\rm BEM}$ with the weakly geometric blocks is asymptotically minimax under $L_Q$ on the multivariate Sobolev ellipsoid $\Theta(\beta,Q)$:
	\begin{align*}
		\sup_{\theta \in \Theta(\beta,Q)} R_Q(\theta,\hat{\theta}_{\rm BEM}) \sim \inf_{\hat{\theta}} \sup_{\theta \in \Theta(\beta,Q)} R_Q(\theta,\hat{\theta}) \sim P(\beta,Q) \varepsilon^{{4 \beta}/{(2 \beta+1)}}
	\end{align*}
	as $\varepsilon \to 0$, where $\inf$ is taken over all estimators and  $P(\beta,Q)$ is given by \eqref{Pdef}.
\end{theorem}
\begin{proof}
From the definition of the weakly geometric blocks and Lemma~3.12 of \cite{Tsybakov}, 
	\begin{align}\label{wgblock}
		|B_1| = \log^2 (1/\varepsilon), \quad J \leq C \log^2 (1/\varepsilon), \quad \max_{1 \leq j \leq J-1} \frac{|B_{j+1}|}{|B_j|} = 1+3 \rho_{\varepsilon},
	\end{align}
	for sufficiently small $\varepsilon$, where $C>0$ is some constant and $\rho_{\varepsilon}=(\log (1/\varepsilon))^{-1}$.
	We take $\varepsilon$ to be sufficiently small so that $|B_j|>p+1$ for every $j$.
	
	For each block $B_j$, the oracle inequality in Corollary~\ref{cor_oracle} yields
	\begin{align*}
		&\sum_{i \in B_j} {\rm E}_{\theta} ((\hat{\theta}_{{\rm BEM}})_i - \theta_i)^{\top} Q ((\hat{\theta}_{{\rm BEM}})_i - \theta_i) \leq \min_{C} \sum_{i \in B_j} {\rm E}_{\theta} ((\hat{\theta}_{C})_i - \theta_i)^{\top} Q ((\hat{\theta}_C)_i - \theta_i) + 2(p+1) {\rm tr} (Q) \varepsilon^2.
	\end{align*}
	Thus,
\begin{align}\label{oracle_ineq}
&R_Q(\theta,\hat{\theta}_{\rm BEM}) = \sum_{j=1}^J \sum_{i \in B_j} {\rm E}_{\theta} ((\hat{\theta}_{{\rm BEM}})_i - \theta_i)^{\top} Q ((\hat{\theta}_{{\rm BEM}})_i - \theta_i) + \sum_{i=N+1}^{\infty} \theta_i^{\top} Q \theta_i \\
		\leq& \sum_{j=1}^J \left( \min_{C} \sum_{i \in B_j} {\rm E}_{\theta} ((\hat{\theta}_{C})_i - \theta_i)^{\top} Q ((\hat{\theta}_C)_i - \theta_i) + 2(p+1) {\rm tr} (Q) \varepsilon^2 \right) + \sum_{i=N+1}^{\infty} \theta_i^{\top} Q \theta_i \nonumber \\
		=& \min_{\bar{C} \in \mathbb{W}_{B,Q}} R_Q(\theta,\hat{\theta}_{\bar{C}}) + 2(p+1) J {\rm tr} (Q) \varepsilon^2. \nonumber
	\end{align}
	
	Also, from Lemma~\ref{lem_block},
	\begin{align}\label{oracle_ineq2}
		\min_{\bar{C} \in \mathbb{W}_{B,Q}} R_Q(\theta,\hat{\theta}_{\bar{C}}) \leq (1+3\rho_{\varepsilon}) \min_{{C} \in \mathbb{W}_{\mathrm{mon},Q}} R_Q(\theta,\hat{\theta}_C) + |B_1| {\rm tr} (Q) \varepsilon^2.
	\end{align}
	
	Finally, consider the linear estimator $\hat{\theta}_{\mathrm{P}}$ in \eqref{mPinsker}, which is asymptotically minimax on $\Theta(\beta,Q)$. 
	By definition, the coefficient matrix $C_{\mathrm{P},i}$ of $\hat{\theta}_{\mathrm{P}}$ is zero for $i > N_1(\kappa) = \lfloor (\kappa^{-1} \lambda_1)^{1/\beta} \rfloor$.
	From Lemma~\ref{lem_kappa2}, $N_1(\kappa) = O(\varepsilon^{-2/(2 \beta + 1)})=o(\varepsilon^{-2})$ as $\varepsilon \to 0$.  
	Therefore, for sufficiently small $\varepsilon$, $C_{\mathrm{P}} \in \mathbb{W}_{\mathrm{mon},Q}$ and
	\begin{align}\label{oracle_ineq3}
		\min_{{C} \in \mathbb{W}_{\mathrm{mon},Q}} R_Q(\theta,\hat{\theta}_C) \leq R_Q(\theta,\hat{\theta}_{\mathrm{P}}).
	\end{align}
	
	Combining \eqref{oracle_ineq}, \eqref{oracle_ineq2} and \eqref{oracle_ineq3} and using \eqref{wgblock}, we obtain
	\begin{align*}
		R_Q(\theta,\hat{\theta}_{\rm BEM}) \leq (1+3\rho_{\varepsilon}) R_Q(\theta,\hat{\theta}_{\mathrm{P}}) + C' \varepsilon^2 \log^2(1/\varepsilon)
	\end{align*}
	for sufficiently small $\varepsilon$, where $C'$ is a constant independent of $\theta$.
	Therefore,
	\begin{align*}
		\sup_{\theta \in \Theta(\beta,Q)} R_Q(\theta,\hat{\theta}_{\rm BEM}) & \leq (1+3\rho_{\varepsilon}) \sup_{\theta \in \Theta(\beta,Q)} R_Q(\theta,\hat{\theta}_{\mathrm{P}}) + C' \varepsilon^2 \log^2(1/\varepsilon) \sim \inf_{\hat{\theta}} \sup_{\theta \in \Theta(\beta,Q)} R_Q(\theta,\hat{\theta})
\end{align*}
	as $\varepsilon \to 0$, where we used Theorem~\ref{th_pinsker}.
	Hence, $\hat{\theta}_{\rm BEM}$ is asymptotically minimax on $\Theta(\beta,Q)$.
\end{proof}

Theorem~\ref{th_bem} indicates that the blockwise Efron--Morris estimator is adaptive not only to smoothness of the unknown function but also to arbitrary quadratic loss.
In other words, it attains sharp adaptive estimation of any linear combinations of the mean sequences simultaneously.
Note that using the blockwise James--Stein estimator for each $j$ only attains adaptive minimaxity for diagonal $Q$.

Theorem~\ref{th_bem} considers the quadratic loss $L_Q$ and parameter space $\Theta(\beta,R)$ with $Q=R$.
Whereas it would be useful if the result can be extended to general $Q$ and $R$, we believe that the current result is an important special case.
Since $L_{aQ}=a L_Q$, Theorem~\ref{th_bem} guarantees that the blockwise Efron--Morris estimator attains asymptotic minimaxity under $L_Q$ on every $\Theta(\beta,aQ)$ simultaneously, in the same way as the blockwise James--Stein estimator attains asymptotic minimaxity under the usual quadratic loss on every Sobolev ellipsoid \cite{Johnstone}.
As discussed after Theorem V.1, simultaneous diagonalizability of $Q$ and $R$ may be essential for the existence of an asymptotically minimax linear estimator on $\Theta(\beta,R)$ under $L_Q$.
We leave extension of Theorem VI.1 to general $Q$ and $R$ for future work.

\section{Simulation results}\label{sec_simulation}
In this section, we present simulation results to compare the performance of several estimators.
We approximated the $Q$-quadratic loss by truncation
\begin{align}
	L_Q (\theta,\hat{\theta}) \approx \sum_{i=1}^n (\hat{\theta}_i-\theta)^{\top} Q (\hat{\theta}_i-\theta) \label{trunc}
\end{align}
for $n$ large enough (see Figure~\ref{risk_n}) and computed the $Q$-quadratic risk by Monte Carlo with 100 repetitions.
We compare three estimators: the multivariate Pinsker estimator $\hat{\theta}_{\mathrm{P}}$ in \eqref{mPinsker}, the blockwise Efron--Morris estimator $\hat{\theta}_{\mathrm{BEM}}$ in \eqref{BEM} and the blockwise James--Stein estimator $\hat{\theta}_{\mathrm{BJS}}$ in \eqref{BJS} applied to each of the $p$ variables individually.
In all experiments, we confirmed that the standard errors of the Monte Carlo estimates are less than 1\%.

Figure~\ref{risk_n} plots the ratio of the $Q$-quadratic risk to the minimax risk $P(\beta,Q) \varepsilon^{4 \beta /(2 \beta+1)}$ (see Theorem~\ref{th_pinsker}) as a function of the truncation point $n$ in \eqref{trunc}, where $p=5$, $\beta=1$, $\varepsilon=10^{-6}$.
The matrix $Q$ is set to $Q=I_p$ (full rank) or $Q=u u^{\top}+10^{-3} I_p$ with $u \sim {\rm N}_p(0,I_p)$ (approximately rank one).
The parameter $\theta$ is fixed to a point on the boundary of $\Theta(\beta,Q)$ that is obtained by using a sample from the least favorable prior $\xi_i \sim {\rm N}_p (0,V_i)$ with $V_i=\varepsilon^2 a_i^{-1}  \kappa^{-1} Q (I_p-a_i  \kappa Q^{-1})_+$ (see the remark after Theorem~\ref{th_pinsker}) as 
\begin{align*}
	\theta = \left( \sum_{i=1}^n a_i^2 \xi_i^{\top} Q^{-1} \xi_i \right)^{-1/2} \xi.
\end{align*}
The figure indicates that $n \approx 10^4$ is enough for the approximation \eqref{trunc} in this setting.
Since $\theta$ was set close to the least favorable values on $\Theta(\beta,Q)$, the $Q$-quadratic risk of the multivariate Pinsker estimator is almost equal to the minimax risk on $\Theta(\beta,Q)$, which is compatible with Theorem~\ref{th_pinsker}.
Also, the $Q$-quadratic risk of the blockwise Efron--Morris estimator is approximately equal to the minimax risk, which is compatible with Theorem~\ref{th_bem}.
On the other hand, the blockwise Jame--Stein estimator fails to attain the minimax risk when $Q$ is close to low-rank (Figure~\ref{risk_n} right).

\begin{figure}[h]
	\centering
	\begin{tikzpicture}
		\begin{axis}[
			title={$Q=I_p$},
			xlabel={$\log_{10} n$}, xmin=3, xmax=6,
			ylabel={risk ratio}, ymin=0, ymax=1.5,
			width=0.45\linewidth
			]
			\addplot[thick, color=black,
			filter discard warning=false, unbounded coords=discard
			] table {
				3.0103    0.1844
				3.3113    0.3498
				3.6124    0.6279
				3.9134    0.9586
				4.2144    1.0231
				4.5154    1.0344
				4.8165    1.0454
				5.1175    1.0566
				5.4185    1.0667
				5.7196    1.0780
				6.0206    1.0892
			};
			\addplot[dashed, thick, color=black,
			filter discard warning=false, unbounded coords=discard
			] table {
				3.0103    0.1828
				3.3113    0.3467
				3.6124    0.6215
				3.9134    0.9454
				4.2144    0.9987
				4.5154    0.9987
				4.8165    0.9987
				5.1175    0.9987
				5.4185    0.9987
				5.7196    0.9987
				6.0206    0.9987
			};
			\addplot[dotted, thick, color=black,
			filter discard warning=false, unbounded coords=discard
			] table {
				3.0103    0.1834
				3.3113    0.3478
				3.6124    0.6237
				3.9134    0.9498
				4.2144    1.0070
				4.5154    1.0108
				4.8165    1.0144
				5.1175    1.0181
				5.4185    1.0215
				5.7196    1.0253
				6.0206    1.0290
			};
		\end{axis}
	\end{tikzpicture} 
	\begin{tikzpicture}
		\begin{axis}[
			title={$Q=u u^{\top}+10^{-3} I_p$},
			xlabel={$\log_{10} n$}, xmin=3, xmax=6,
			ylabel={risk ratio}, ymin=0, ymax=1.5,
			width=0.45\linewidth
			]
			\addplot[thick, color=black,
			filter discard warning=false, unbounded coords=discard
			] table {
				3.0103    0.1083
				3.3113    0.2104
				3.6124    0.3952
				3.9134    0.6943
				4.2144    0.9941
				4.5154    1.0138
				4.8165    1.0206
				5.1175    1.0275
				5.4185    1.0334
				5.7196    1.0399
				6.0206    1.0460
			};
			\addplot[dashed, thick, color=black,
			filter discard warning=false, unbounded coords=discard
			] table {
				3.0103    0.1076
				3.3113    0.2093
				3.6124    0.3928
				3.9134    0.6894
				4.2144    0.9847
				4.5154    0.9976
				4.8165    0.9976
				5.1175    0.9976
				5.4185    0.9976
				5.7196    0.9976
				6.0206    0.9976
			};
			\addplot[dotted, thick, color=black,
			filter discard warning=false, unbounded coords=discard
			] table {
				3.0103    0.1359
				3.3113    0.2905
				3.6124    0.5811
				3.9134    1.0234
				4.2144    1.3928
				4.5154    1.4085
				4.8165    1.4107
				5.1175    1.4131
				5.4185    1.4151
				5.7196    1.4174
				6.0206    1.4194
			};
		\end{axis}
	\end{tikzpicture} 
	\caption{Ratio of $Q$-quadratic risk to the minimax risk as a function of $n$. dashed: multivariate Pinsker estimator, solid: blockwise Efron--Morris estimator, dotted: blockwise James--Stein estimator}
	\label{risk_n}
\end{figure}
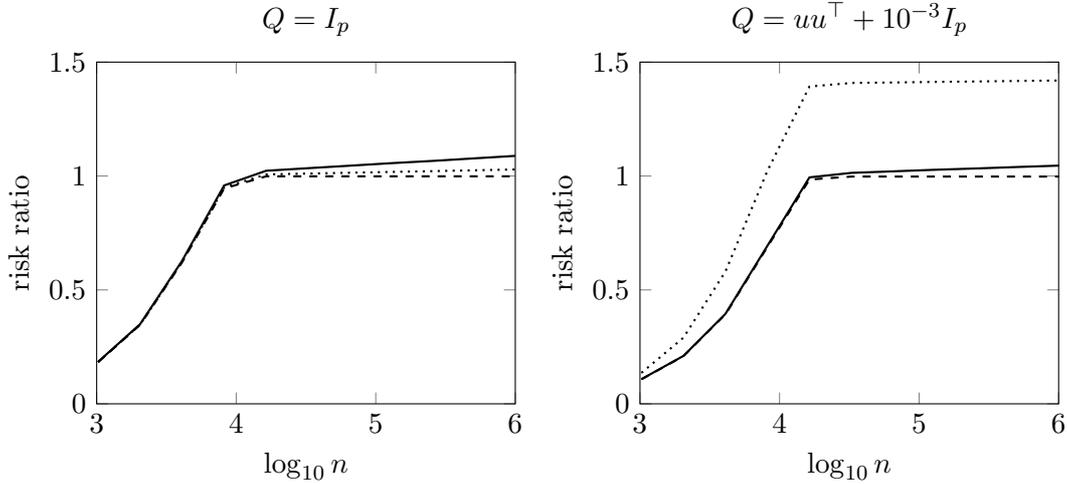

Figure~\ref{minimax_check} plots the ratio of $Q$-quadratic risk to the minimax risk $P(\beta,Q) \varepsilon^{4 \beta /(2 \beta+1)}$ as a function of $\varepsilon$, where $p=5$, $\beta=1$ and $n=10^6$.
The setting of $Q$ and $\theta$ are the same with Figure~\ref{risk_n}.
It shows the asymptotic minimaxity of the multivariate Pinsker estimator and the blockwise Efron--Morris estimator as $\varepsilon \to 0$.

\begin{figure}[h]
	\centering
	\begin{tikzpicture}
		\begin{axis}[
			title={$Q=I_p$},
			xlabel={$\log_{10} \varepsilon$}, xmin=-8, xmax=-4,
			ylabel={risk ratio}, ymin=0, ymax=1.5,
			width=0.45\linewidth
			]
			\addplot[thick, color=black,
			filter discard warning=false, unbounded coords=discard
			] table {
				-8.0000    1.0027
				-7.0000    1.0175
				-6.0000    1.0906
				-5.0000    1.4504
				-4.0000    3.0512
			};
			\addplot[dashed, thick, color=black,
			filter discard warning=false, unbounded coords=discard
			] table {
				-8.0000    1.0000
				-7.0000    1.0005
				-6.0000    0.9994
				-5.0000    0.9992
				-4.0000    0.9978
				-3.0000    1.0056
				-2.0000    1.0169
			};
			\addplot[dotted, thick, color=black,
			filter discard warning=false, unbounded coords=discard
			] table {
				-8.0000    1.0010
				-7.0000    1.0064
				-6.0000    1.0297
				-5.0000    1.1510
				-4.0000    1.6868
				-3.0000    3.9011
				-2.0000   11.2988
			};
		\end{axis}
	\end{tikzpicture} 
	\begin{tikzpicture}
		\begin{axis}[
			title={$Q=u u^{\top}+10^{-3} I_p$},
			xlabel={$\log_{10} \varepsilon$}, xmin=-8, xmax=-4,
			ylabel={risk ratio}, ymin=0, ymax=1.5,
			width=0.45\linewidth
			]
			\addplot[thick, color=black,
			filter discard warning=false, unbounded coords=discard
			] table {
				-8.0000    1.0006
				-7.0000    1.0034
				-6.0000    1.0213
				-5.0000    1.1101
				-4.0000    1.5108
				-3.0000    3.2099
				-2.0000    9.1207
			};
			\addplot[dashed, thick, color=black,
			filter discard warning=false, unbounded coords=discard
			] table {
				-8.0000    0.9999
				-7.0000    0.9997
				-6.0000    1.0003
				-5.0000    1.0001
				-4.0000    1.0019
				-3.0000    0.9958
				-2.0000    0.9634
			};
			\addplot[dotted, thick, color=black,
			filter discard warning=false, unbounded coords=discard
			] table {
				-8.0000    1.4374
				-7.0000    1.4385
				-6.0000    1.4382
				-5.0000    1.4592
				-4.0000    1.5401
				-3.0000    1.9626
				-2.0000    3.7694
			};
		\end{axis}
	\end{tikzpicture} 
	\caption{Ratio of $Q$-quadratic risk to the minimax risk as a function of $\varepsilon$. dashed: multivariate Pinsker estimator, solid: blockwise Efron--Morris estimator, dotted: blockwise James--Stein estimator}
	\label{minimax_check}
\end{figure}
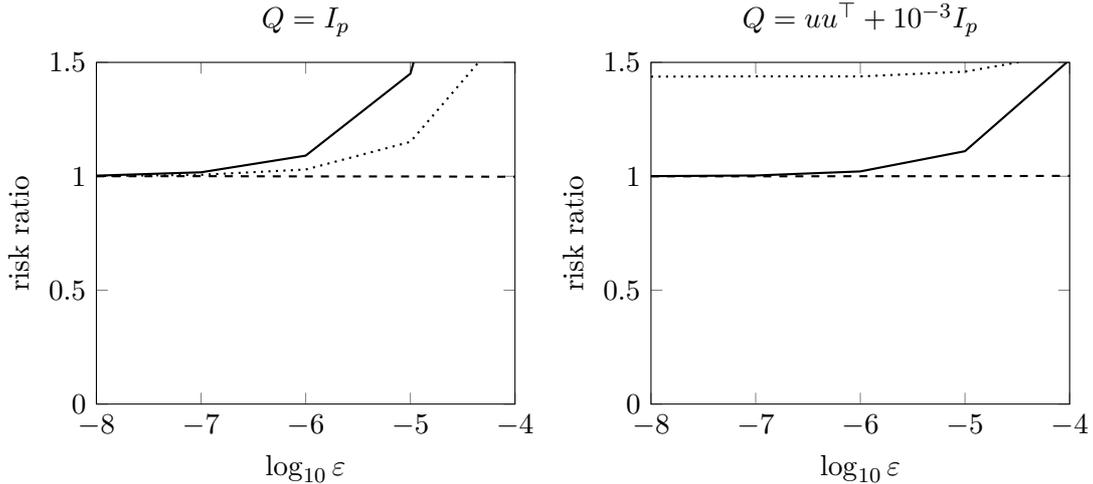

Figure~\ref{denoise} visualizes an example in terms of function denoising by using the Fourier expansion similarly to Figure~\ref{sobolev_sample}, where $p=2$, $\varepsilon=10^{-6}$, $n=10^6$, $\beta=1$  and
\begin{align*}
	Q = \begin{pmatrix} 5 & -2 \\ -2 & 1 \end{pmatrix}.
\end{align*}
The setting of $\theta$ is the same with Figure~\ref{risk_n}.
In this case, the $Q$-quadratic loss of the blockwise Efron--Morris estimator and blockwise James--Stein estimator are $9.64 \times 10^{-8}$ and $1.01 \times 10^{-7}$, respectively, and the minimax risk $P(\beta,Q) \varepsilon^{4 \beta /(2 \beta+1)}$ is $9.53 \times 10^{-8}$.

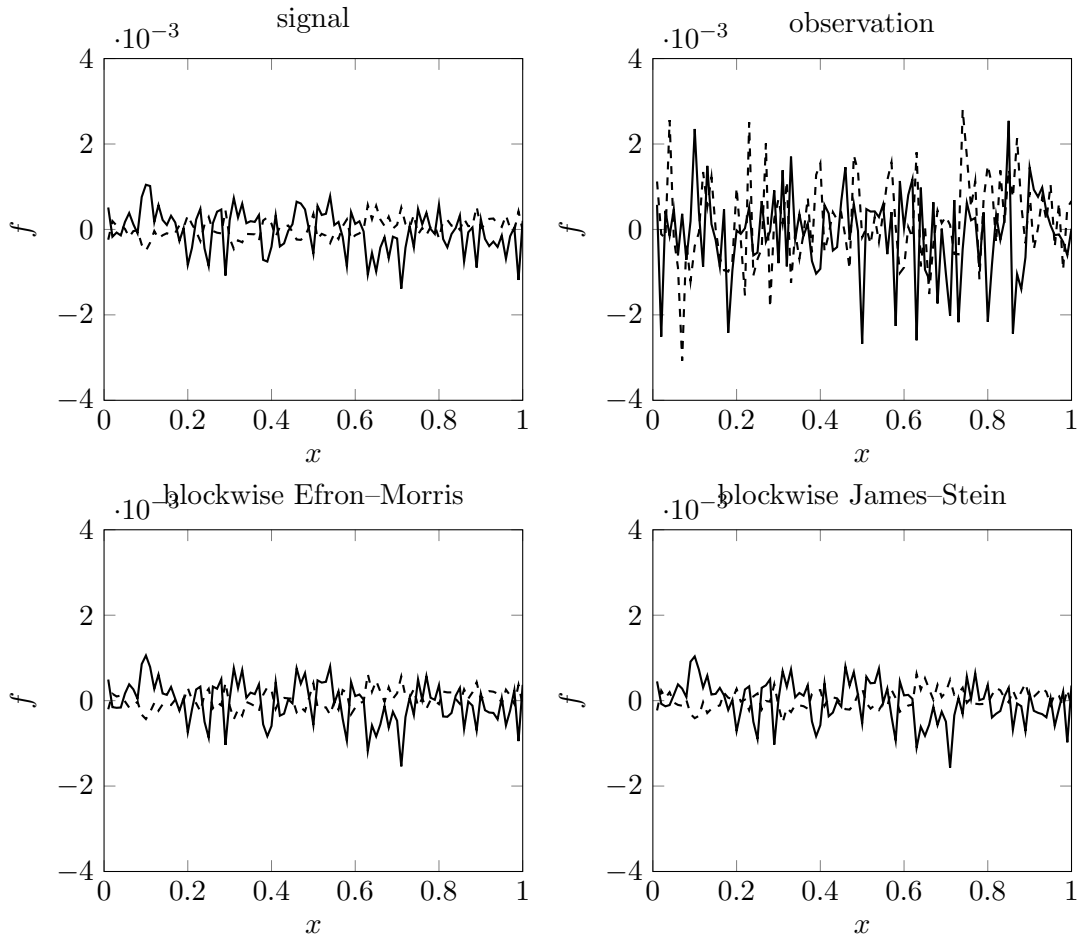
\begin{figure}[h]
	\begin{tikzpicture}
		\begin{axis}[
			title={signal},
			xlabel={$x$}, xmin=0, xmax=1, ylabel={$f$}, ymin=-0.004, ymax=0.004, 
			width=0.45\linewidth
			]
			\addplot[ thick, color=black,
			filter discard warning=false, unbounded coords=discard
			] table [x=x, y=f1] {signal.dat};
			\addplot[dashed,  thick, color=black,
			filter discard warning=false, unbounded coords=discard
			] table [x=x, y=f2] {signal.dat};
		\end{axis}
	\end{tikzpicture} 
	\begin{tikzpicture}
		\begin{axis}[
			title={observation},
			xlabel={$x$}, xmin=0, xmax=1, ylabel={$f$},  ymin=-0.004, ymax=0.004,
			width=0.45\linewidth
			]
			\addplot[ thick, color=black,
			filter discard warning=false, unbounded coords=discard
			] table [x=x, y=f1] {obs.dat};
			\addplot[dashed,  thick, color=black,
			filter discard warning=false, unbounded coords=discard
			] table [x=x, y=f2] {obs.dat};
		\end{axis}
	\end{tikzpicture} \\
	\begin{tikzpicture}
		\begin{axis}[
			title={blockwise Efron--Morris},
			xlabel={$x$}, xmin=0, xmax=1, ylabel={$f$},  ymin=-0.004, ymax=0.004, 
			width=0.45\linewidth
			]
			\addplot[ thick, color=black,
			filter discard warning=false, unbounded coords=discard
			] table [x=x, y=f1] {blockEM.dat};
			\addplot[dashed,  thick, color=black,
			filter discard warning=false, unbounded coords=discard
			] table [x=x, y=f2] {blockEM.dat};
		\end{axis}
	\end{tikzpicture} 
	\begin{tikzpicture}
		\begin{axis}[
			title={blockwise James--Stein},
			xlabel={$x$}, xmin=0, xmax=1, ylabel={$f$},  ymin=-0.004, ymax=0.004,
			width=0.45\linewidth
			]
			\addplot[ thick, color=black,
			filter discard warning=false, unbounded coords=discard
			] table [x=x, y=f1] {blockJS.dat};
			\addplot[dashed,  thick, color=black,
			filter discard warning=false, unbounded coords=discard
			] table [x=x, y=f2] {blockJS.dat};
		\end{axis}
	\end{tikzpicture} 
	\caption{Example of function denoising. solid: $f_1$, dashed: $f_2$}
	\label{denoise}
\end{figure} 

\section{Conclusion}\label{sec_concl}

In this study, we introduced a multivariate version of the Gaussian sequence model and showed that the blockwise Efron--Morris estimator attains adaptive minimaxity over the multivariate Sobolev ellipsoids.
Notably, it adapts to arbitrary quadratic loss in addition to smoothness of the unknown function.
This adaptivity comes from the singular value shrinkage in the Efron--Morris estimator, which is different from the scalar shrinkage in the James--Stein estimator and thresholding-type shrinkage commonly used in sparse estimation.
By shrinking the singular values towards zero separately, the Efron--Morris estimator successfully captures the correlations between multiple sequences and utilize them for adaptive estimation.
We expect that the singular value shrinkage is useful in other estimation problems as well.

In addition to adaptive estimation over the Sobolev ellipsoids, many results have been obtained for estimation in the Gaussian sequence model, which corresponds to $p=1$ in our setting \cite{Johnstone}.
It is an interesting future work to further investigate generalization of these results to multivariate Gausian sequence models.
For example, since a penalized version of the blockwise James--Stein estimator attains adaptive minimaxity over general ellipsoids and hyperrectangles \cite{Cavalier01}, such a modification may be possible for the blockwise Efron--Morris estimator as well.
Generalization to the linear inverse problem \cite{Cavalier02}, which corresponds to the Gaussian sequence model where the observation variance increases with indices, is another direction for future research.
Also, it is an interesting problem whether existing methods for adaptation to inhomogenous smoothness in the Besov spaces, such as wavelet thresholding \cite{Donoho94,Donoho95} and deep neural networks \cite{Suzuki}, can be extended to attain adaptation to arbitrary quadratic loss.
There are several extensions of the Efron--Morris estimator \citep{Tsukuma} and they may be also useful.

Empirical Bayes methods have been found to be effective in adaptive estimation \cite{Johnstone04,Zhang}.
The blockwise Efron--Morris estimator provides another example of this, since the Efron--Morris estimator is naturally interpreted as an empirical Bayes estimator \cite{Efron72}.
Recently, a superharmonic prior for the normal mean matrix parameter has been developed \cite{Matsuda}, which shrinks the singular values towards zero and can be viewed as a generalization of Stein's prior \cite{Stein74}.
The generalized Bayes estimator with respect to this prior is minimax under the matrix quadratic loss and has qualitatively the same behavior with the Efron--Morris estimator \cite{Matsuda22}.
It is an interesting future work to investigate Bayesian inference in multivariate Gaussian sequence models with respect to such singular value shrinkage priors, including uncertainty quantification.

\begin{appendix}
\subsection{Proof of Proposition~\ref{prop_sobolev}}
Let $f \in W (\beta,L)$.
Since $\phi_1' \equiv 0$, $\phi_{2m}'=-(2 \pi m) \phi_{2m+1}$ and $\phi_{2m+1}'=(2 \pi m) \phi_{2m}$ for $m \geq 1$,
\begin{align}\label{fourier_deriv}
	f_j^{(l)} (x) = \sum_{m=1}^{\infty} (2 \pi m)^{l} (s_{2m,j}(l) \phi_{2m}(x) + s_{2m+1,j}(l) \phi_{2m+1}(x)), \end{align}
for $l=0,1,2,\dots$, where 
\begin{align*}
	(s_{2m,j}(l),s_{2m+1,j}(l))= \begin{cases}
		(\theta_{2m,j}, \theta_{2m+1,j}), & l \equiv 0 \mod 4, \\ 
		(\theta_{2m+1,j}, -\theta_{2m,j}), & l \equiv 1 \mod 4, \\ 
		(-\theta_{2m,j}, -\theta_{2m+1,j}), & l \equiv 2 \mod 4, \\ 
		(-\theta_{2m+1,j}, \theta_{2m+1,j}), & l \equiv 3 \mod 4,
	\end{cases}
\end{align*}
for $m=1,2,\dots$.
Thus, from the orthonormality of $\phi=(\phi_i)$,
\begin{align*}
	\int_0^1 f_j^{(\beta)} (x) f_k^{(\beta)} (x) {\rm d} x =  \sum_{m=1}^{\infty} (2 \pi m)^{2 \beta} (s_{2m,j} s_{2m,k} + s_{2m+1,j} s_{2m+1,k})= \pi^{2 \beta} \sum_{i=1}^{\infty} a_{\beta,i}^2 \theta_{ij} \theta_{ik}.
\end{align*}
Therefore,
\begin{align*}
	1 \geq \int_0^1 f^{(\beta)} (x)^{\top} L^{-2} f^{(\beta)} (x) {\rm d} x  &= \sum_{j=1}^p \sum_{k=1}^p (L^{-2})_{jk} \int_0^1 f_j^{(\beta)} (x) f_k^{(\beta)} (x) {\rm d} x \\
	&= \pi^{2 \beta} \sum_{j=1}^p \sum_{k=1}^p (L^{-2})_{jk} \sum_{i=1}^{\infty} a_{\beta,i}^2 \theta_{ij} \theta_{ik} = \pi^{2 \beta} \sum_{i=1}^{\infty} a_{\beta,i}^2 \theta_i^{\top} L^{-2} \theta_i.
\end{align*}
which indicates that $\theta \in \Theta ( \beta, L^2/\pi^{2 \beta})$.

Conversely, let $\theta \in \Theta \left( \beta, L^2/\pi^{2 \beta} \right)$.
Since 
\begin{align*}
	\sum_{m=1}^{\infty} m^l (|\theta_{2m,j}| + |\theta_{2m+1,j}|) & \leq \sum_{m=1}^{\infty} m^{\beta-1} (|\theta_{2m,j}| + |\theta_{2m+1,j}|) \\
	& \leq \left( 2 \sum_{m=1}^{\infty} m^{2 \beta} (\theta_{2m,j}^2 + \theta_{2m+1,j}^2) \right) \left( \sum_{m=1}^{\infty} m^{-2} \right) < \infty,
\end{align*}
for $j=1,\dots,p$ and $l=0,1,\dots,\beta-1$, the function $f_j^{(l)}$ is expressed by the uniformly convergent series \eqref{fourier_deriv} and thus periodic for $j=1,\dots,p$ and $l=0,1,\dots,\beta-1$.
Let
\begin{align*}
	g_j(x) = \sum_{m=1}^{\infty} (2 \pi m)^{\beta} (s_{2m,j}(\beta) \phi_{2m}(x) + s_{2m+1,j}(\beta) \phi_{2m+1}(x)), \quad j=1,\dots,p.
\end{align*}
By using the orthonormality of $\phi=(\phi_i)$ similarly to the above,
\begin{align*}
	\int_0^1 g (x)^{\top} L^{-2} g (x) {\rm d} x = \pi^{2 \beta} \sum_{i=1}^{\infty} a_{\beta,i}^2 \theta_i^{\top} L^{-2} \theta_i \leq 1.
\end{align*}
It also implies that $g_j \in L_2[0,1]$ for $j=1,\dots,p$.
Since the Fourier series of a function in $L_2[0,1]$ is termwise integrable,
\begin{align*}
	\int_a^b g_j (x) {\rm d} x &= \sum_{m=1}^{\infty} (2 \pi m)^{\beta} \left( s_{2m,j}(\beta) \int_a^b \phi_{2m}(x)  {\rm d} x + s_{2m+1,j}(\beta) \int_a^b \phi_{2m+1}(x) {\rm d} x \right) \\
	&= \sum_{m=1}^{\infty} (2 \pi m)^{\beta-1} \left( s_{2m,j}(\beta-1) (\phi_{2m}(b)-\phi_{2m}(a))  + s_{2m+1,j}(\beta)(\phi_{2m+1}(b)-\phi_{2m+1}(a)) \right) \\
	&= f_j^{(\beta-1)} (b) - f_j^{(\beta-1)} (a),
\end{align*}
for $0 \leq a \leq b \leq 1$, which indicates that $f_j^{(\beta-1)}$ is absolutely continuous and $g_j = f_j^{(\beta)}$ almost everywhere for $j=1,\dots,p$.
	Therefore, $f \in W (\beta,L)$.
	
	\subsection{Technical Lemma}\label{app_lem}
	\begin{lemma}\label{lem_bayes_risk}
		Let $Y \sim {\rm N}_p (\theta, \varepsilon^2 I_p)$.
		For the prior $\theta \sim {\rm N}_p(0,\Sigma)$, the Bayes estimator of $\theta$ under the quadratic loss $l_Q(\theta,\hat{\theta})=(\hat{\theta}-\theta)^{\top} Q (\hat{\theta}-\theta)$ is the posterior mean
		\begin{align*}
			\hat{\theta}^{\pi} = \Sigma (\Sigma + \varepsilon^2 I_p)^{-1} y,
		\end{align*}
		which is independent of $Q$, and its Bayes risk is
		\begin{align*}
			{\rm E}_{\pi} {\rm E}_{\theta} (\hat{\theta}^{\pi} - \theta)^{\top} Q (\hat{\theta}^{\pi} - \theta) = \varepsilon^2 {\rm tr} (Q \Sigma (\Sigma + \varepsilon^2 I_p)^{-1}).
		\end{align*}
	\end{lemma}
	\begin{proof}
		Straightforward calculation. See \cite{Lehmann}.
	\end{proof}
	
	\begin{lemma}\label{lem_chi2}
		Let $X_i \sim {\rm N} (0,1)$ be independent standard Gaussian random variables for $i=1,\dots,n$. 
		Then, for $t > 0$,
		\begin{align*}
			{\rm Pr} \left[ \sum_{i=1}^n b_i^2 (X_i^2 - 1) \geq t \sum_{i=1}^n b_i^2 \right] \leq \exp \left( -\frac{t^2 \sum_{i=1}^n b_i^2}{8 \max_{1 \leq i \leq n} b_i^2} \right).
		\end{align*}
	\end{lemma}
	\begin{proof}
		See Lemma 3.5 of \cite{Tsybakov}.
	\end{proof}
	
	\begin{lemma}\label{lem_concentration}
		Let $X_i \sim {\rm N} (0,1)$ be independent standard Gaussian random variables for $i=1,\dots,n$. 
		Then,
		\begin{align*}
			{\rm E} \left[ \left( \sum_{i=1}^n b_i^2 X_i^2 \right)^2 \right] \leq 3 \left( \sum_{i=1}^n b_i^2 \right)^2.
		\end{align*}
	\end{lemma}
	\begin{proof}
		Since ${\rm E}[X_i^4]=3$ and ${\rm E} [ X_i^2 X_j^2 ] = 1$ for $i \neq j$,
		\begin{align*}
			{\rm E} \left[ \left( \sum_{i=1}^n b_i^2 X_i^2 \right)^2 \right] &= \sum_{i=1}^n b_i^4 {\rm E} [ X_i^4 ] + \sum_{i=1}^n \sum_{j=1}^n I(i \neq j) b_i^2 b_j^2 {\rm E} [ X_i^2 X_j^2 ] \\
			&= 3 \sum_{i=1}^n b_i^4 + \sum_{i=1}^n \sum_{j=1}^n I(i \neq j) b_i^2 b_j^2 \\
			&\leq 3 \left( \sum_{i=1}^n b_i^2 \right)^2.
		\end{align*}
	\end{proof}
	
	\subsection{Finite-dimensional case}\label{app_finite}
	Here, we extend the finite-dimensional version of Pinsker's theorem \cite{Nussbaum99,Wasserman} to multivariate setting.
	
	Suppose $y_i = \theta_i + \varepsilon \xi_i$ for $i=1,\dots,n$, where $\xi_i \sim {\rm N}_p (0, (1/n) I_p)$ are independent.
	We consider estimation of $\theta=(\theta_1,\dots,\theta_n)$ under the quadratic loss specified by a $p \times p$ positive definite matrix $Q$:
	\begin{align*}
		L_Q(\theta,\hat{\theta}) =  \sum_{i=1}^{n} (\hat{\theta}_i - \theta_i)^{\top} Q (\hat{\theta}_i - \theta_i).
	\end{align*}
	The risk function is
	\begin{align*}
		R_Q(\theta,\hat{\theta}) = {\rm E}_{\theta} \left[ \sum_{i=1}^{n} (\hat{\theta}_i - \theta_i)^{\top} Q (\hat{\theta}_i - \theta_i) \right].
	\end{align*}
	We consider asymptotically minimax estimation on the ellipsoid
	\begin{align*}
		\Theta_n(Q) = \left\{ \theta = (\theta_1,\dots,\theta_n) \relmiddle| \sum_{i=1}^{n} \theta_i^{\top} Q^{-1} \theta_i \leq 1 \right\} \subset \mathbb{R}^{np},
	\end{align*}
	as $n \to \infty$.
	
	Let $g(\kappa)=\varepsilon^2 \kappa^{-1} {\rm tr} (I_p- \kappa Q^{-1})_+$ for $\kappa>0$, where $( \cdot )_+$ denotes the projection onto the positive semidefinite cone defined in Section~\ref{sec_pinsker}.
	Since $g$ is monotone decreasing, $g(\kappa) \to \infty$ as $\kappa \to 0$ and $g(\kappa) \to 0$ as $\kappa \to \infty$, there exists a unique $\kappa$ that satisfies $g(\kappa)=1$.
By using this $\kappa$, we define a linear estimator $\hat{\theta}_{\mathrm{P}}=(\hat{\theta}_{\mathrm{P},i})$ by
	\begin{align*}
		\hat{\theta}_{\mathrm{P},i} = C_{\mathrm{P}} y_i, \quad C_{\mathrm{P}} = (I_p - \kappa Q^{-1})_+.\end{align*}

	\begin{theorem}\label{th_finite}
		The estimator $\hat{\theta}_{\mathrm{P}}  = (\hat{\theta}_{\mathrm{P},i})$
is asymptotically minimax on $\Theta_n(Q)$:
		\begin{align}
			\sup_{\theta \in \Theta_n(Q)} R_Q(\theta,\hat{\theta}_{\mathrm{P}}) &\sim \inf_{\hat{\theta}} \sup_{\theta \in \Theta_n(Q)} R_Q(\theta,\hat{\theta}) \sim \varepsilon^2 {\rm tr} (C_{\mathrm{P}} Q) (= O(1)),
		\end{align}
		as $n \to \infty$, where $\inf$ is taken over all estimators.
	\end{theorem}
	\begin{proof}
		Let $Q=U^{\top} \Lambda U$ be a spectral decomposition of $Q$ with $\Lambda={\rm diag}(\lambda_1,\dots,\lambda_p)$. 
		Then, $C_{\mathrm{P}} = U^{\top} (I_p-\kappa \Lambda^{-1})_+ U$.
		Therefore, by using $(1-(1-l)_+)^2 \leq l^2$,
		\begin{align}
			(I_p-C_{\mathrm{P}})Q(I_p-C_{\mathrm{P}}) &= U^{\top} (I_p-(I_p-\kappa \Lambda^{-1})_+)^2 \Lambda U \preceq U^{\top} \cdot \kappa^2 \Lambda^{-2} \cdot \Lambda U =\kappa^2 Q^{-1}. \label{mat_ineq2}
		\end{align}
		Also, by using $\kappa \lambda_j^{-1} + (1-\kappa \lambda_j^{-1})_+=1$ when $(1-\kappa \lambda_j^{-1})_+ > 0$,
		\begin{align}\label{lem_risk2}
			{\rm tr} (C_{\mathrm{P}} Q (\kappa Q^{-1} + C_{\mathrm{P}})) &= \sum_{j=1}^p (1-\kappa \lambda_j^{-1})_+ \lambda_j (\kappa \lambda_j^{-1} + (1- \kappa \lambda_j^{-1})_+) \nonumber \\
			&= \sum_{j=1}^p (1- \kappa \lambda_j^{-1})_+ \lambda_j =  {\rm tr} (C_{\mathrm{P}} Q).
		\end{align}
		
From
		\begin{align*}
			{\rm E}_{\theta} \left[ (\hat{\theta}_{\mathrm{P},i}-\theta_i) (\hat{\theta}_{\mathrm{P},i}-\theta_i)^{\top} \right] = (I_p-C_{\mathrm{P}}) \theta_i \theta_i^{\top} (I_p-C_{\mathrm{P}}) + n^{-1} \varepsilon^2 C_{\mathrm{P}}^2,
		\end{align*}
		we have
		\begin{align*}
			R_Q(\theta,\hat{\theta}_{\mathrm{P}}) &= \sum_{i=1}^{n} \theta_i^{\top} (I_p-C_{\mathrm{P}}) Q (I_p-C_{\mathrm{P}}) \theta_i + \varepsilon^2 {\rm tr} (C_{\mathrm{P}}^2 Q) \\
			&\leq \kappa^2 \sum_{i=1}^{n} \theta_i^{\top} Q^{-1} \theta_i + \varepsilon^2 {\rm tr} (C_{\mathrm{P}}^2 Q) \\
			&\leq \kappa^2 \cdot \varepsilon^2 \kappa^{-1} {\rm tr} (C_{\mathrm{P}}) + \varepsilon^2 {\rm tr} (C_{\mathrm{P}}^2 Q) \\
			& = \varepsilon^2 {\rm tr} (C_{\mathrm{P}} Q (\kappa Q^{-1}+C_{\mathrm{P}}))  = \varepsilon^2 {\rm tr} (C_{\mathrm{P}} Q)
		\end{align*}
		for $\theta \in \Theta_n(Q)$, where we used \eqref{mat_ineq2}, $g(\kappa)=1$ and \eqref{lem_risk2}.
		Thus,
		\begin{align}\label{finite_upper}
			\inf_{\hat{\theta}} \sup_{\theta \in \Theta_n(Q)} {\rm E}_{\theta} \left[ \sum_{i=1}^{n}  (\hat{\theta}_{\mathrm{P},i} - \theta_i)^{\top} Q (\hat{\theta}_{\mathrm{P},i} - \theta_i) \right] \leq \varepsilon^2 {\rm tr} (C_{\mathrm{P}} Q).
		\end{align}
		
		Let $\delta \in (0,1)$ and $\pi$ be the prior $\theta_i \sim {\rm N}(0,\delta^2 V)$ independently with $V=(\varepsilon^2/n) \kappa^{-1} Q (I_p-\kappa Q^{-1})_+$ for $i=1,\dots,n$.
		Then, the Bayes estimator with respect to this prior (posterior mean) is
		\begin{align*}
			\hat{\theta}^{\pi}_i = C^{\pi} x_i, \quad C^{\pi} = \delta^2 V ( \delta^2 V + (\varepsilon^2/n) I_p)^{-1},
		\end{align*}
		for $i=1,\dots,n$ and its Bayes risk is
		\begin{align*}
			r(\pi) & = n \cdot \delta^2 (\varepsilon^2/n) {\rm tr} (Q V (\delta^2 V + (\varepsilon^2/n) I_p)^{-1}) \\
			& \geq \delta^2 \varepsilon^2 {\rm tr} (Q V (V + (\varepsilon^2/n) I_p)^{-1}) =  \delta^2 \varepsilon^2 {\rm tr} (C_{\mathrm{P}} Q).
		\end{align*}
		Let $\mathbb{E}$ be the class of estimators $\hat{\theta}$ of ${\theta}$ that take values in $\Theta_n(Q)$.
		Since $\Theta_n(Q)$ is closed and convex, we can restrict our attention to estimators in $\mathbb{E}$ in considering the lower bound for minimax risk over $\Theta_n(Q)$.
		For any estimator $\hat{\theta} \in \mathbb{E}$,
		\begin{align*}
			r(\pi) &\leq \int R_Q(\theta,\hat{\theta}) \pi(\theta) {\rm d} \theta \leq \sup_{\theta \in \Theta_n(Q)} R_Q(\theta,\hat{\theta}) + \sup_{\hat{\theta} \in \mathbb{E}} \int_{\Theta_n(Q)^c} R_Q(\theta,\hat{\theta}) \pi(\theta) {\rm d} \theta,
		\end{align*}
		where $\Theta_n(Q)^c=\mathbb{R}^{np} \setminus \Theta_n(Q)$.
		Taking the infimum over all estimators in $\mathbb{E}$,
		\begin{align*}
			r(\pi) & \leq \inf_{\hat{\theta}} \sup_{\theta \in \Theta_n(Q)} R_Q(\theta,\hat{\theta}) + \sup_{\hat{\theta} \in \mathbb{E}} \int_{\Theta_n(Q)^c} R_Q(\theta,\hat{\theta}) \pi(\theta) {\rm d} \theta,
		\end{align*}
		which yields
		\begin{align*}
			\inf_{\hat{\theta}} \sup_{\theta \in \Theta_n(Q)} R_Q(\theta,\hat{\theta}) \geq r(\pi) - \sup_{\hat{\theta} \in \mathbb{E}} \int_{\Theta_n(Q)^c} R_Q(\theta,\hat{\theta}) \pi(\theta) {\rm d} \theta.
		\end{align*}
		Now, from $(a+b)^{\top} Q (a+b) \leq 2 (a^{\top} Q a + b^{\top} Q b)$ and the Cauchy-Schwarz inequality,
		\begin{align*}
			& \sup_{\hat{\theta} \in \mathbb{E}} \int_{\Theta_n(Q)^c} R_Q(\theta,\hat{\theta}) \pi(\theta) {\rm d} \theta \\
			\leq & 2 \int_{\Theta_n(Q)^c} \left( \sum_{i=1}^n \theta_i^{\top} Q \theta_i \right) \pi(\theta) {\rm d} \theta + 2 \sup_{\hat{\theta} \in \mathbb{E}} \int_{\Theta_n(Q)^c} {\rm E}_{\theta} \left[ \sum_{i=1}^n \hat{\theta}_i^{\top} Q \hat{\theta}_i \right] \pi(\theta) {\rm d} \theta \\
			\leq & 2 (\pi(\Theta_n(Q)^c))^{1/2} {\rm E}_{\pi} \left[ \left( \sum_{i=1}^n \theta_i^{\top} Q \theta_i \right)^2 \right]^{1/2} + 2 M \pi (\Theta_n(Q)^c),
		\end{align*}
		where
		\begin{align*}
			M = \max_{\theta \in \Theta_n(Q)} \sum_{i=1}^n \theta_i^{\top} Q \theta_i = \lambda_1^2 = O(1).
		\end{align*}
		Let $\xi_i = \delta^{-1} V^{-1/2} \theta_i \sim {\rm N}_p(0,I_p)$ and $V^{1/2} Q^{-1} V^{1/2} = U^{\top} \Psi U$ with $\Psi={\rm diag}(\psi_{1},\dots,\psi_{p})$ be a spectral decomposition of $V^{1/2} Q^{-1} V^{1/2}$.
		Then,
		\begin{align*}
			\theta_i^{\top} Q^{-1} \theta_i = \delta^2 \xi_i^{\top} V^{1/2} Q^{-1} V^{1/2} \xi_i = \delta^2 (U \xi_i)^{\top} \Psi (U \xi_i) = \delta^2 \sum_{j=1}^p \psi_{j} (U {\xi}_i)_j^2.
		\end{align*}
		Therefore, by putting $b_{j}^2 = \delta^2 \psi_{j}$,
		\begin{align*}
			\pi (\Theta_n(Q)^c) & = {\rm Pr} \left[ \sum_{i=1}^n \theta_i^{\top} Q^{-1} \theta_i > 1 \right] = {\rm Pr} \left[ \sum_{i=1}^n \sum_{j=1}^p b_{j}^2 ((U \xi_i)_j^2-1) > \frac{1-\delta^2}{\delta^2} \sum_{i=1}^n \sum_{j=1}^p b_{j}^2  \right],
		\end{align*}
		where we used $\sum_{i,j} b_{j}^2 = n \delta^2 {\rm tr} (V^{1/2} Q^{-1} V^{1/2}) = \delta^2 g(\kappa) = \delta^2$.
		Since $U \xi_i \sim {\rm N}_p (0,I_p)$ independently for $i=1,\dots,n$, from Lemma~A.2,
		\begin{align*}
			\pi (\Theta_n(Q)^c) = O \left( \exp \left( - \frac{(1-\delta^2)^2}{8 \delta^4} \frac{\sum_{i,j} b_{j}^2}{\max_{j} \{ b_{j}^2 \mid b_j^2>0 \}} \right) \right) = O \left( \exp \left( - C n  \right) \right)\end{align*}
		for $\delta \in (1/\sqrt{2},1)$ with some constant $C>0$.
		Also, by putting $V^{1/2} Q V^{1/2} = U^{\top} D U$ with $D= {\rm diag}(\mu_{1},\dots,\mu_{p})$,
		\begin{align*}
			\theta_i^{\top} Q \theta_i = \delta^2 \xi_i^{\top} V^{1/2} Q V^{1/2} \xi_i = \delta^2 (U \xi_i)^{\top} D (U \xi_i) = \delta^2 \sum_{j=1}^p \mu_{j} (U \xi_i)_j^2.
		\end{align*}
		Thus, from Lemma~A.3,
		\begin{align*}
			{\rm E}_{\pi} \left[ \left( \sum_{i=1}^n \theta_i^{\top} Q \theta_i \right)^2 \right] &\leq 3 \delta^4  \left( \sum_{i=1}^n \sum_{j=1}^p \mu_{j} \right)^2 = 3 \delta^4 n^2 {\rm tr} (Q V)^2 = 3 \delta^4 \varepsilon^4 \kappa^{-1}  {\rm tr} (C_{\mathrm{P}} Q^2) = O(1). \end{align*}
Therefore, from ${\rm tr} (C_{\mathrm{P}} Q)=O(1)$,
		\begin{align*}
			\inf_{\hat{\theta}} \sup_{\theta \in \Theta_n(Q)} R_Q(\theta,\hat{\theta}) \geq \delta^2 \varepsilon^2 {\rm tr} (C_{\mathrm{P}} Q) (1+o(1)).
		\end{align*}
		By taking $\delta \to 1$, 
		\begin{align}\label{finite_lower}
			\inf_{\hat{\theta}} \sup_{\theta \in \Theta_n(Q)} R_Q(\theta,\hat{\theta}) \geq \varepsilon^2 {\rm tr} (C_{\mathrm{P}} Q) (1+o(1)).
		\end{align}
		
		From \eqref{finite_upper} and \eqref{finite_lower}, we obtain the theorem.
	\end{proof}
	\end{appendix}
		
	\section*{Acknowledgments}
	The author thanks the referees for helpful comments.
	The author thanks Iain Johnstone and William Strawderman for fruitful discussion.
	This work was supported by JSPS KAKENHI Grant Numbers 19K20220, 21H05205, 22K17865 and JST Moonshot Grant Number JPMJMS2024.

\end{document}